\documentclass[11pt, reqno]{amsart}
\usepackage[utf8]{inputenc}
\usepackage{amsfonts}
\usepackage{hyperref}

\hypersetup{}
\usepackage{amsmath}
\usepackage{amsthm}
\usepackage{pdflscape}
\usepackage{pgfplots}
\usepackage{mathrsfs}
\usepackage{mathtools}
\usepackage[top=2.5cm, bottom=2.5cm, left=2.2cm, right=2.2cm]{geometry}
\usepackage{amssymb,bbm}
\usepackage{shuffle}
\usepackage{enumerate}
\usepackage{enumitem}
\usepackage[initials]{amsrefs}
\usepackage{stmaryrd}
\usepackage{appendix}
\usepackage{tikz}
\usepackage{tikz-cd}

\pgfplotsset{compat = 1.18}

\usepackage[nodayofweek]{datetime}

\usepackage{verbatim}

\usepackage{graphicx}
\usepackage{matlab-prettifier}

\numberwithin{equation}{section}
\numberwithin{figure}{section}

\theoremstyle{plain}
\newtheorem{theorem}{Theorem}[section]
\newtheorem{lemma}[theorem]{Lemma}
\newtheorem{proposition}[theorem]{Proposition}
\newtheorem{corollary}[theorem]{Corollary}

\newtheorem{definition}[theorem]{Definition}

\newtheorem{remark}[theorem]{Remark}

\newtheorem{conjecture}{Conjecture}
\newtheorem*{conjecture*}{Conjecture}

\allowdisplaybreaks

\newcommand{\bbE}{\mathbb{E}}
\newcommand{\bbF}{\mathbb{F}}

\newcommand{\bbH}{\mathbb{H}}
\newcommand{\bbI}{\mathbb{I}}

\newcommand{\bbN}{\mathbb{N}}

\newcommand{\bbP}{\mathbb{P}}

\newcommand{\R}{\mathbb{R}} % Special bc of frequent use

\newcommand{\bbT}{\mathbb{T}}

\newcommand{\bbZ}{\mathbb{Z}}

\newcommand{\mC}{\mathcal{C}}

\newcommand{\mF}{\mathcal{F}}

\newcommand{\mH}{\mathcal{H}}

\newcommand{\mL}{\mathcal{L}}

\newcommand{\mN}{\mathcal{N}}

\newcommand{\mP}{\mathcal{P}}

\newcommand{\mX}{\mathcal{X}}

\newcommand{\bfB}{\mathbf{B}}

\newcommand{\X}{\mathbf{X}} % Frequent use
\newcommand{\Y}{\mathbf{Y}} % Frequent use
\newcommand{\bfZ}{\mathbf{Z}}

\newcommand{\bfs}{\mathbf{s}}

\newcommand{\x}{\mathbf{x}}

\newcommand{\sfA}{\mathsf{A}}

\newcommand{\sfD}{\mathsf{D}}

\newcommand{\sfF}{\mathsf{F}}

\newcommand{\sfI}{\mathsf{I}}
\newcommand{\sfJ}{\mathsf{J}}
\newcommand{\sfK}{\mathsf{K}}
\newcommand{\sfL}{\mathsf{L}}

\newcommand{\sfP}{\mathsf{P}}
\newcommand{\sfQ}{\mathsf{Q}}
\newcommand{\sfR}{\mathsf{R}}
\newcommand{\sfS}{\mathsf{S}}

\newcommand{\sfU}{\mathsf{U}}
\newcommand{\sfV}{\mathsf{V}}

\newcommand{\eps}{\varepsilon}
\newcommand{\kap}{\kappa}

\newcommand{\mPS}{{\mathcal{Q}}}

\newcommand{\Mat}{\mathsf{M}}
\newcommand{\lfe}{\rho}
\newcommand{\mlfe}{\mu}
\newcommand{\intPot}{W}
\newcommand{\stat}{\pi}
\newcommand{\Cn}{\bbZ_n}
\newcommand{\ta}{\tilde{\alpha}}
\newcommand{\tb}{\tilde{\beta}}
\newcommand{\ff}{\mathfrak{f}}

\newcommand{\mSpace}{\mathscr{M}}
\newcommand{\Ric}{\sfV}
\newcommand{\bSigma}{\mathbf{\Sigma}}
\newcommand{\bPhi}{\mathbf{\Phi}}

\title{A case study on the long-time behavior of the Gaussian local-field equation}
\author{Kevin Hu}
\address{Kevin Hu: Division of Applied Mathematics, Brown University.}
\email{\href{mailto:kevin_hu@brown.edu}{kevin\_hu@brown.edu}}

\author{Kavita Ramanan}
\address{Kavita Ramanan: Division of Applied Mathematics, Brown University.}
\email{\href{mailto:kavita_ramanan@brown.edu}{kavita\_ramanan@brown.edu}}
\date{\today}

\begin{document}
\begin{abstract}
For any integer $\kappa \geq 2$, the $\kappa$-local-field equation ($\kappa$-LFE) characterizes the limit of the neighborhood path empirical measure of interacting diffusions on $\kappa$-regular random graphs, as the graph size goes to infinity.  
It has been conjectured that the long-time behavior of the (in general non-Markovian) $\kappa$-LFE coincides with that of a certain more tractable Markovian analog, the Markov $\kappa$-local-field equation.  
In the present article, we prove this conjecture for the case when $\kappa = 2$ and the diffusions are one-dimensional with affine drifts. 
As a by-product of our proof, we also show that for interacting diffusions on the $n$-cycle (or 2-regular random graph on $n$ vertices), the limits $n \rightarrow \infty$ and $t\rightarrow \infty$ commute. Along the way, we also establish well-posedness of the Markov $\kappa$-local field equations with affine drifts for all $\kappa \geq 2$, which may be of independent interest. 
\end{abstract}
\maketitle
\noindent \textbf{Key words:} interacting diffusions; Ornstein-Uhlenbeck process; local-field equation, conditional McKean-Vlasov equation, Markov local-field equation, long-time behavior, sparse free energy. \\
\noindent \textbf{MSC 2020 subject classifications:}  Primary 60K35, 60J60, 60G15; Secondary  82C22, 82C31

\section{Introduction}

\subsection{Motivation}
\label{s:mot}
Fix $n \in \bbN$ and let $\Cn = (V_n, E_n)$ denote the cycle graph on $n$ elements, that is,
$$V_n = \big\{0, \ldots, n-1\big\}, \quad E_n = \big\{ (u, v) \in V_n \times V_n: u = v \text{ mod }n  \big\}.$$ 
Let $U, W: \R \rightarrow \R$ be continuously differentiable functions and consider the following system of interacting diffusions:
\begin{equation}
    dZ^n_v(t) = - \bigg( U'\big(Z^n_v(t)\big) + \sum_{(u, v) \in E_n} W'\big(Z^n_v(t) - Z^n_{u}(t)\big) \bigg) dt + \sqrt{2} dB_v(t), \quad v \in \Cn,
    \label{e:zFinite}
\end{equation}
where $\{B_v\}_{v \in \Cn}$ is a family of identical and independently distributed (i.i.d.) 1-dimensional Brownian motions. Under general conditions on $U'$ and $W'$, the equation \eqref{e:zFinite} is well-posed. The system \eqref{e:zFinite} is a simple example of an interacting particle system on a sparse graph. Such strongly interacting systems arise as models in physics, biology, economics, and many other fields (see \cite{ramanan2023ICM} for a review). Moreover, equations of the form \eqref{e:zFinite} are finite-difference approximations of stochastic partial differential equations, for example, see Section 2.1 of \cite{hairer2009introduction} or Chapter 1 of \cite{berglund2022spde}.

As a special case of Theorem 3.3 of  \cite{lacker2023localweakconvergence} it follows that on finite time intervals, as $n \rightarrow \infty$ the process $Z^n$ described in \eqref{e:zFinite} converges in distribution (in the local topology) to the unique solution of the following infinite-dimensional diffusion:
\begin{equation}
    dZ_v(t) = - \Big[ U'\big(Z_v(t)\big) + W'\big(Z_v(t) - Z_{v+1}(t)\big) + W'\big(Z_v(t) - Z_{v-1}(t)\big) \Big] dt + \sqrt{2} dB_v(t), \quad v \in \bbZ.
    \label{e:zInfinite}
\end{equation}
Each $Z_v$ takes values in the space $\mC$ of real-valued continuous functions on $[0, \infty)$. The work \cite{lacker2023localweakconvergence} (see Theorem 3.7 therein) also established a hydrodynamic limit, which showed that the neighborhood empirical measure $\mu_n := \frac{1}{n} \sum_{v = 1}^n\delta_{(Z_{v-1}^n. Z_v^n, Z_{v+1}^n)}$, a random probability measure on $\mC^3$, converges weakly as $n \rightarrow \infty$ to the marginal law of $\{Z_v\}_{v \in \{0, \pm1 \}}$. 

Furthermore, it follows from Corollary 3.14 of \cite{lacker2023marginal} that the marginal law of $\{Z_v\}_{v \in \{0, \pm1\}}$ is the unique solution to an autonomous equation known as the \emph{local-field equation} (henceforth abbreviated LFE), given by
\begin{equation}
\label{e:introLFE}
\begin{aligned}
    dZ_0(t) &= - \Big( U'(Z_0(t)) + \sum_{v \in \{\pm1\}} W'\big(Z_0(t) - Z_v(t)\big) \bigg) dt + \sqrt{2} dB_0(t), \\
    dZ_v(t) &= - \Gamma\big(t, Z_v[t], Z_0[t]\big) dt + \sqrt{2} dB_v(t), \quad v = \pm 1,
\end{aligned}
\end{equation}
where $\omega[t] = \{\omega(s)\}_{s \in [0, t]}$ denotes the history of the path $\omega$ up to time $t$, and $\Gamma$ takes the form of the following conditional expectation:
\begin{equation}
    \Gamma(t, \omega, \omega') = \mathbb{E}\bigg[  U'(Z_0(t)) + \sum_{v \in \{\pm1\}} W'\big(Z_0(t) - Z_v(t)\big)\,\bigg|\, Z_0[t] = \omega[t], Z_1[t] = \omega'[t]\bigg], 
    \label{e:introGamma}
\end{equation}
for $\omega, \omega' \in \mC$. Thus, at any time $t \in (0, \infty)$, the functional $\Gamma(t, \cdot, \cdot)$ depends on both the trajectories of the process as well as the law of the process up to that time through certain conditional distributions. We note that the interaction potential in \eqref{e:zFinite} does not scale with the population size $n$ and therefore the hydrodynamic limit described above is different in nature to the one in \cite{GPZ1988}, which involves an additional diffusion-type spatial scaling.

For sparse graphs the LFE plays a  role analogous to mean-field equations, also referred to as \emph{McKean-Vlasov equations}, which characterizes both the limit marginal law of a typical particle and limit empirical measure of interacting diffusions on the complete graph (or suitably dense graphs), as the size of the graph goes to infinity \cite{sznitman1991poc, CD2022poc}. Both the LFE and McKean-Vlasov equations describe nonlinear stochastic processes, namely those whose instantaneous evolutions depend on their laws in addition to their states. However, the LFE \eqref{e:introLFE} involves a conditional nonlinearity since it depends on the law through a conditional distribution,  and thus is an example of a \emph{conditional McKean-Vlasov equation}, henceforth abbreviated to CMVE. The long-time behavior of McKean-Vlasov equations with a nonlinearity that is affine in the measure argument has been the subject of many classical works such as \cite{carrillo2003kinetic, CGM2008prob, tugaut2013doublewell}, as well as more recent developments such as \cite{barbu2023evolution, carrillo2020long, delgadino2023phase, GWWZ2022uniform}. However, due to the lack of continuity of conditional expectations with respect to the underlying measure as well as the conditioning variable, the analysis of CMVEs is more challenging \cite{buckdahn2023cmve}. In particular, results on the long-time behavior of CMVE are limited (though see the introduction of \cite{hu2024htheorem} for a review). 

The solution of the LFE is in general a non-Markovian process. Indeed, due to the history dependence in the conditional expectation in \eqref{e:introGamma}, the LFE \eqref{e:introLFE} is a 
functional or path-dependent stochastic differential equation (SDE). The LFE can also be interpreted as a singular infinite-dimensional stochastic evolution equation with degenerate noise in the sense of Da Prato and Zabczyk \cite{daPrato1992infinite}, see also \cites{huang2019nfpk, wang2022functional}. It is well known that the long-time behavior of infinite-dimensional equations with degenerate noise can be subtle and requires different approaches, as introduced for example in \cite{huang2019nfpk, HMS2011}. We emphasize that our setting poses additional challenges; the lack of continuity of the conditional expectation implies that the LFE \eqref{e:introLFE} does not fall into the framework of \cite{huang2019nfpk}.

A major obstacle to the analysis and simulation of the LFE is its non-Markovian nature; as such, it is of interest to obtain a suitable Markovian approximation of the LFE. The recent work \cite{hu2024htheorem} studies the long-time behavior of a related Markovian CMVE known as the \emph{Markov local-field equation} (henceforth abbreviated MLFE), which takes the form:
\begin{equation}
\label{e:introMLFE}
\begin{aligned}
    dX_0(t) &= - \Big( U'(X_0(t)) + \sum_{v \in \{\pm1\}} W'\big(X_0(t) - X_v(t)\big) \bigg) dt + \sqrt{2} dB_0(t), \\
    dX_v(t) &= - \gamma\big(t, X_v(t), X_0(t)\big) dt + \sqrt{2} dB_v(t), \quad v = \pm 1,
\end{aligned}
\end{equation}
where $\gamma:\R_+ \times \R \times \R \rightarrow \R$ is a measurable function that satisfies
\begin{equation}
    \gamma(t, X_0(t), X_1(t)) := \mathbb{E}\bigg[ U'\big(X_0(t)\big) + \sum_{v \in \{\pm 1\}} W'\big(X_0(t) - X_v(t)\big) \,\bigg|\, X_0(t), X_1(t) \bigg], \quad x, y \in \R.
    \label{e:introM_Gamma}
\end{equation}
Despite superficial similarities between the LFE \eqref{e:introLFE}-\eqref{e:introGamma} and the MLFE \eqref{e:introMLFE}-\eqref{e:introM_Gamma}, the relationship between solutions to the two equations is difficult to characterize; in particular, there is no obvious connection between the MLFE and the infinite particle system \eqref{e:zInfinite}. 

In this work, we take a first step towards exploring the connection between the LFE and MLFE for diffusions, in both dynamical and equilibrium settings. We make the following conjecture:

\begin{conjecture}
\label{c:LFE_MLFE}
For $t \geq 0$, let $\rho_t$ and $\mu_t$ denote the laws of the solutions to the LFE and MLFE, respectively. Then for ergodic interacting diffusions with drift of gradient form, we have $\lim_{t \rightarrow \infty} d(\rho_t, \mu_t) = 0$ for a suitable metric $d$ on $\mP(\R^3)$.
\end{conjecture}

If true, the conjecture would reduce the analysis of the long-time behavior of the non-Markovian LFE to that of the more tractable MLFE. The recent work \cite{hu2024htheorem} establishes well-posedness of the MLFE for bounded potentials, and then provides a thorough analysis of the long-time behavior of the MLFE for a class of interacting diffusions on $\kap$-regular trees, with continuously differentiable potentials. In particular, it shows that as $t \rightarrow \infty$, the solution of the corresponding MLFE converges to a fixed point of a certain cavity equation on the $\kap$-regular tree, and that there exists a unique fixed point when $\kap = 2$.

It is worth noting that a few works \cite{cocomello2023exact,cocomello2025exact, ganguly2022thesis, GanRam25Stationary} have examined the relationship between the MLFE and LFE in the context of interacting jump processes on finite state spaces.  For example, for a certain class of jump processes that include the SIR and SEIR processes, it has been shown in \cite{cocomello2023exact, cocomello2025exact} that the LFE in fact coincides with the MLFE. Furthermore, for a class of reversible jump processes, it is shown in \cite{ganguly2022thesis, GanRam25Stationary} that any stationary distribution of the LFE is also a stationary distribution for the MLFE.  However, the setting of diffusions presents different technical challenges from the jump process setting.  For example, while in the jump case the MLFE is described by a nonlinear ordinary differential equation for which uniqueness is relatively easy to establish,  the MLFE is characterized by an irregular nonlinear partial differential equation (PDE) whose well-posedness is not immediate (see the discussion in Section \ref{ss:intro-main}).  Also,  unlike in the jump case, the range of influence of the evolution of a diffusion at any given vertex become instantaneously infinite as soon as $t > 0$, making control of the correlations between diffusions at different vertices more delicate.  Nevertheless, in the next section we focus on a particular setting in which we can provide an affirmative response to Conjecture 1.

\subsection{Gaussian setting and main contributions}\label{ss:intro-main} We focus on the setting where $U$ and $W$ are the following quadratic potentials:
\begin{equation}
    U(x) = \frac{\alpha + \beta}{2}|x|^2, \quad W(x) = - \frac{\beta}{4}|x|^2, \quad x \in \R,
    \label{e:potentials}
\end{equation}
for $\alpha, \beta \in \R$. In this case, the process \eqref{e:zFinite} corresponds to the following Ornstein-Uhlenbeck process
\begin{equation}
    dZ^n_v(t) = - \bigg( \alpha Z^n_v(t) + \frac{\beta}{2}\big(Z^n_{v+1}(t) + Z^n_{v-1}(t) \big) \bigg) dt + \sqrt{2} dB_v(t), \quad v \in V_n.
    \label{e:zOU}
\end{equation}
If the initial conditions are normally distributed, the solutions to the finite dimensional SDE \eqref{e:zOU}, the infinite-dimensional SDE \eqref{e:zInfinite}, the LFE \eqref{e:introLFE}, and the MLFE \eqref{e:introMLFE} are all given by Gaussian processes, which are amenable to direct analysis. Therefore in this special case, we will refer to the LFE as the Gaussian LFE (or GLFE) and the MLFE as the Gaussian MLFE (or GMLFE). Since well-posedness of the MLFE was proved in \cite{hu2024htheorem} only when $W'$ is bounded, we first establish well-posedness of the GMLFE in Theorem \ref{t:r:wp}. In fact, we establish it in the more general setting of interacting Gaussian diffusions on $\kap$-regular trees for $\kap \in \bbN$ and $\kap \geq 2$, which is introduced in Section \ref{ss:kap_MLFE}.

The purpose of this work is two-fold. Our first goal is to establish the convergence (locally in distribution) of $Z^n$ from \eqref{e:zFinite} to the corresponding $Z$ defined in \eqref{e:zInfinite} over the \emph{infinite} time horizon when $(U, W)$ are of the form \eqref{e:potentials}, thus generalizing the finite-time convergence result of Theorem 3.3 of \cite{lacker2023localweakconvergence}. When $\alpha > |\beta|$, classical results (e.g., see Exercise 5.6.18 of \cite{karatzas1991book} or Section 2 of \cite{jordan1998variational}) imply that \eqref{e:zOU} converges to the unique invariant probability measure $\pi^n$ on $\R^{V_n}$, which is given by
\begin{equation}
    \pi^n(dx_0, \ldots dx_{n-1}) = \exp\bigg(-\sum_{v \in V_n} \bigg(\frac{\alpha + \beta}{2}|x_v|^2 - \frac{\beta}{8}\sum_{(u, v) \in E_n} |x_v - x_u|^2 \bigg)\bigg) \prod_{v \in V_n} dx_v.
    \label{e:pi_n}
\end{equation}
We identify in Theorem \ref{t:r:lfeConv} a Gaussian measure $\stat$ on $\R^3$ such that any solution $\rho_t$ to the GLFE, starting from a Gaussian initial condition, converges as $t \rightarrow \infty$ to $\pi$ in total variation. Moreover, we show that this rate of convergence is exponential.
\begin{figure}[h]
    \[
\begin{tikzcd}[column sep=6 cm, row sep=3 cm]
\big\{Z^n_v(t)\big\}_{v \in \{-1, 0, 1\}} \arrow[d, "{\parbox{3cm}{\centering \footnotesize{Theorem 3.19 \cite{lacker2023marginal}} \\ $n \rightarrow \infty$}}", squiggly]  \arrow[ r, "\text{Classical, }t \rightarrow \infty ", dashed] & \stat^n \arrow[d, "\parbox{3 cm}{\centering \footnotesize{Remark \ref{r:pi_n_to_pi}} \\ $n \rightarrow \infty$}"]  \\
\big\{Z_v(t)\big\}_{v \in \{-1, 0, 1\}}  \arrow[r, "\text{Theorem \ref{t:r:lfeConv},\, } t \rightarrow \infty"] & \stat  
\end{tikzcd}
\]
\caption{Interchange of $n \rightarrow \infty$ and $t \rightarrow \infty$ limits for the interacting diffusion $Z^n$ in \eqref{e:zOU}. The dashed convergence follows from classical results for Langevin diffusions described above, and the squiggly convergence follows from Theorem 3.19 of \cite{lacker2023marginal}. The solid convergence follow from the results of this paper; the right solid arrow is a consequence of Theorem \ref{t:r:lfeConv}, and the downward solid arrow follows from Remark \ref{r:pi_n_to_pi}.}
  \label{f:commutative}
\end{figure}
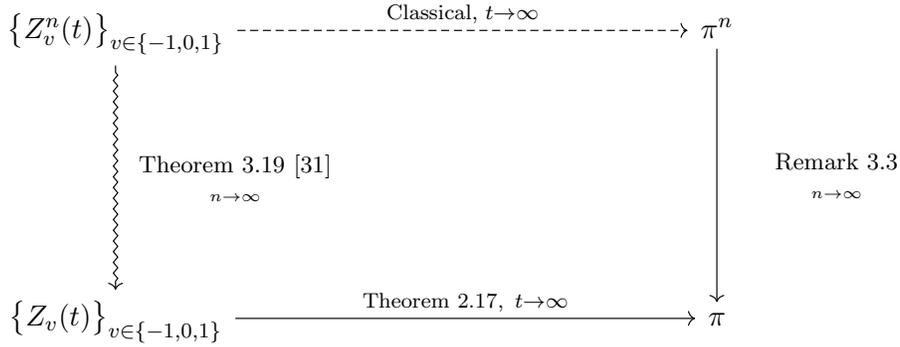

The second goal is to provide evidence for Conjecture \ref{c:LFE_MLFE} by verifying it in the case of quadratic potentials \eqref{e:potentials}. We obtain in Theorem 2.19 an exponential rate of convergence of the solution to the GMLFE in the total variation distance to the invariant measure $\pi$. When combined with the aforementioned convergence results for the GLFE, this implies that total variation distance between solutions to the GLFE and GMLFE decreases exponentially fast (see Corollary \ref{c:r:lfeMlfe} for more details).

\subsection{Generalizations and Future Directions}\label{s:gen} It is natural to ask whether results analgous to those obtained for quadratic potentials can be generalized to more general $( U, W)$. One possible approach is to identify a Lyapunov function on the space of probability measures for the measure-flow of solutions to the LFE. In the recent work \cite{hu2024htheorem}, a Lyapunov function was identified for the measure-flow of the $\kap$-MLFE for a large class of potentials $(U, W)$. This functional, referred to as the \emph{sparse free energy} in \cite{hu2024htheorem}, is given by 
\begin{equation}
    \nu \mapsto \bbH_2(\nu) := \mH(\nu|\pi) - \mH(\bar{\nu}|\bar{\pi}),
    \label{e:sfe'}
\end{equation}
where $\nu$ is a probability measure on $\R^3$, $\bar{\nu}$ is the $(0, 1)$-marginal of $\nu$ (see Definition \ref{def:not:marginal}), and $\pi$ is the unique invariant measure of the MLFE. We believe that $\bbH_2$ is also a Lyapunov function for the measure-flow of the LFE.

\begin{conjecture}
    The sparse free energy $\bbH_2$ serves as a Lyapunov function for the measure-flow of the LFE.
\end{conjecture}

Note that this would imply Conjecture 1, but is a much stronger statement that is not implied by that conjecture.

\subsection{Notation} We now collect some standard notation used throughout the paper.
\label{s:notation}

\subsubsection{Matrices} \label{ss:vec-matrices} 
For $d \in \bbN$, we use $\R^{d \times d}$ to denote the set of $d \times d$ matrices and let $\mathsf{I}_d$ denote the $d$-dimensional identity matrix. For $\mathsf{B} \in \R^{d \times d}$, we let $\mathsf{B}_{i, j}$ denote the $(i,j)$-th entry of $\mathsf{B}$. We write $\mathsf{B}^*$ for the Hermitian conjugate of $
\mathsf{B}$ (or the transpose of $\mathsf{B}$ when $\mathsf{B}$ is real-valued). We use $\delta_{i, j}$ to denote the usual Kronecker delta. Let $$ \|\mathsf{B}\| := \max_{i, j \in \{0, \ldots, \kap\}} |\mathsf{B}_{i,j}|, \quad \mathsf{B} \in \R^{(1 + \kap) \times (1 + \kap)},$$ denote the $l^\infty$ norm on $\R^{(1 + \kap) \times (1 + \kap)}$. We use $\|\cdot\|_F : \R^{d \times d} \rightarrow [0, \infty)$ to denote the Frobenius norm, $\det: \R^{d \times d} \rightarrow \R$ to denote the determinant, and $\text{tr}:\R^{d \times d} \rightarrow \R$ to denote the trace. We say a matrix $\sfA \in \R^{d \times d}$ is non-degenerate if $\det(\sfA) \neq 0$.

\subsubsection{Probability measures and equipped functionals} Given a Polish space $\mX$ denote a Polish space, we use $\mP(\mX)$ to denote the set of Borel probability measures on $\mX$ equipped with the topology of weak convergence.

Let $(\Omega, \mF, \bbP)$ be a probability space. For a random variable $X:(\Omega, \mF) \rightarrow \mX$, we write $\mL(X) \in \mP(\mX)$ for the law of $X$ under $\bbP$. For a measure $\nu \in \mP(\mX)$, we let $Y$ denote the canonical random variable with law $\nu$, and for any measurable function $f: \mX \rightarrow \R^d$, we write
\begin{equation*}
    \bbE^\nu[f(Y)] := \int_\mX f(y) \nu(dy).
\end{equation*}
We write $X \stackrel{(d)}{=} Y$ if $X$ and $Y$ are two random variables that are equal in distribution.  

Let $d \in \bbN$. For $m \in \R^d$ and any positive semi-definite matrix $\Sigma \in \R^{d \times d}$, we write $\lambda \sim \mathcal{N}(m, \Sigma)$ to mean that $\lambda$ is the Gaussian measure on $\R^d$ with mean $m$ and covariance $\Sigma$.  Given a Gaussian measure $\lambda \in \mP(\R^d)$, we let $\Sigma_\lambda \in \R^{d \times d}$ denote its covariance matrix. We say $\lambda$ is \emph{centered} if $m = 0$ and \emph{non-degenerate} if $\Sigma_\lambda$ is strictly positive definite.

\subsubsection{Function spaces} \label{ss:function-spaces} For $T \in (0, \infty)$ and $d \in \bbN$, we let $\mC^d_T$ denote the space of continuous functions from $[0, T]$ to $\R^d$ with the topology of uniform convergence. Similarly, let $\mC^d$ denote the space of continuous functions from $[0, \infty)$ to $\R^d$ equipped with the topology of uniform convergence on compact sets. For a path $\omega \in \mC^d$, we let $\omega[T] := \{\omega_s\}_{s \in [0, T]}$ denote the restriction of $\omega$ to $[0, T]$. For a probability measure $\mu \in \mP(\mC^d)$, we let $\mu[T]$ denote the image of $\mu$ under the (measurable) projection map $\mC^d \mapsto \mC_T^d$ given by $\omega \mapsto \omega[T]$. Similarly we let $\mu_t$ denote the time-marginal of $\mu$ at $t$, that is the image of $\mu$ under the measurable map $\omega \mapsto \omega(t)$.

\section{Local-Field Equations}
\subsection{Preliminaries about probability measures on $\kap$-stars}  {Let $\mX$ be a Polish space. For $\kap \in \bbN$ with $\kap \geq 2$, we will consider vectors of the form $\mathbf{x} = (x_0, x_{\{1, \ldots, \kap\}}) \in \mX^{1 + \kap}$ where we distinguish the element $x_0$. To maintain the intuition described in the introduction, in the special case $\kap = 2$ we will instead write $x_{-1}$ for $x_2$, so that $\mathbf{x} = (x_{-1}, x_0, x_1)$. Recall the function space notation from Section \ref{ss:function-spaces}.} We start by defining a set of probability measures that are invariant with respect to the automorphisms of the star graph. Throughout, we fix $\kap \in \bbN$ with $\kap \geq 2$.
\begin{definition}[Edge marginal]
\label{def:not:marginal}
Let $\mathfrak{p}:\mX^{1 + \kap} \rightarrow \mX \times \mX$ be the projection map $\mathfrak{p}(x_0, x_1, \ldots, x_\kap) = (x_0, x_1)$. For any $\nu \in \mP(\mX^{1 + \kap})$, define the \emph{edge marginal} $$\bar{\nu} := \nu \circ \mathfrak{p}^{-1} \in \mP(\mX \times \mX),$$ to be the 0-1 marginal of $\nu$.     
\end{definition}
\begin{definition}[Symmetric probability measures]
\label{def:not:symProb}
    The space $\mPS_{\kap}$  is the set of probability measures $\nu \in \mP(\R^{1 + \kap})$ that satisfy the following two properties:
    \begin{enumerate}
        \item (Leaf Exchangeability) For any permutation $\tau$ on $\{1, \ldots, \kap\}$, we have
        \begin{equation}
            \label{e:sym1}
            \nu(dx_0, dx_{\{1, \ldots, \kap\}}) = \nu(dx_0, dx_{\{\tau(1), \ldots, \tau(\kap)\}}).
        \end{equation}
        \item (Edge Symmetry) The edge marginal $\bar{\nu}$ satisfies
        \begin{equation}
                    \label{e:sym2}
            \bar{\nu}(dx_0, dx_1) = \bar{\nu}(dx_1, dx_0).
        \end{equation}
    \end{enumerate}
    Additionally,  we define the space $\mPS_{\kap}^\infty$ to be the set of measures $\mu \in \mP(\mC^{1 + \kap})$ such that $\mu_t \in \mPS_{\kap}$ for all $t\in [0, \infty)$. 
\end{definition}
\begin{remark}[Notation for conditional distributions]\normalfont \label{r:conditonalDistributions} {When $\nu \in \mP(\R^d)$ is absolutely continuous with respect to $d$-dimensional Lebesgue measure, we abuse notation and denote the density $\tfrac{d\nu}{dx}$ by $\nu$ so that $\nu(dx) = \nu(x) dx$.}
    Also, for a measure $\lambda \in \mPS_2$ with $\Y = (Y_{-1}, Y_0, Y_1) \sim \lambda$, we will often use $\lambda(y_{-1}|y_0, y_1) := \lambda(y_{-1}, y_0, y_1)/\bar{\lambda}(y_0, y_1)$ to denote the density of the conditional distribution under $\lambda$ of $Y_{-1}$ given $Y_0$ and $Y_1$ and let $\lambda_0$ denote any 1-marginal distribution of $\lambda$, recalling that \eqref{e:sym1} ensures that they are all identical.
\end{remark}

We identify the set of covariance matrices that respect the symmetries of $\mPS_\kap$. 

\begin{definition}[Covariance matrix parametrization]
\label{d:mSpace}
    Define the map $\Mat_{\kap}: \R^3 \rightarrow \R^{(1 + \kap) \times (1 + \kap)}$ by 
\begin{equation}
\Big( \Mat_\kap (a, b, c) \Big)_{i,j} = \begin{cases}
    a &\quad \text{if } i = j, \\
    b &\quad \text{if }i \neq j, i = 0 \text{ or } j = 0, \\
    c &\quad \text{otherwise, }
\end{cases}
\label{e:Mat}
\end{equation}
for $i, j \in \{0, \ldots, \kap\}$. In matrix form, it is written as
\begin{equation}
\label{e:AltMat}
    \Mat_\kap(a, b, c) = \begin{bmatrix}
        a & b & b &\cdots & b \\
        b & a &  c & \cdots & c \\
        b & c & a & \cdots & c \\
        \vdots & \vdots & \vdots & \ddots & \vdots \\
        b & c & c & \cdots & a
    \end{bmatrix}.
\end{equation}
Define the set of matrices
\begin{equation}
    \mSpace_\kap := \big \{ \Mat_\kap(a, b, c): (a,b, c) \in \R^3\big\}.
\end{equation}
Notice that $\mSpace_\kap$ is a closed vector space and $\Mat_\kap:\R^3 \rightarrow \mSpace_\kap$ is a continuous vector space isomorphism. We define $\mSpace_\kap^+$ to be the set of positive definite matrices in $\mSpace_\kap$. 
{Observe} that $\mSpace_\kap^+$ is an open set.
\end{definition}

The following lemma establishes the relationship between $\mSpace_\kap$ and Gaussian measures in $\mPS_\kap$. Lemma \ref{l:gSym} is proved in Appendix \ref{ap:ref_l:gsym}.

\begin{lemma}[Correspondence between measures and covariance matrices]
\label{l:gSym}
    Let $\kap \in \bbN$ satisfy $\kap \geq 2$. If $\lambda \in \mPS_\kap$ is a centered Gaussian measure, then its covariance matrix $\Sigma_\lambda$ lies in $\mSpace_\kap$. Conversely, if $\Sigma = \Mat_\kap(a, b, c)$ is positive semi-definite for $(a, b, c) \in \R^3$, then $\lambda \sim \mathcal{N}(0, \Sigma)$ satisfies $\lambda \in \mPS_\kap$.
\end{lemma}

\subsection{Markov local-field equation on the $\kappa$-star} 
\subsubsection{Definition of the Markov local-field equation}
\label{ss:kap_MLFE}
\begin{definition}[$\kappa$-regular Markov local-field equation]
\label{d:mlfe}
    Let $\kap \in \bbN$, $\lambda \in \mP(\R^{1 + \kap})$  and let $U, W: \R \rightarrow \R$ be continuously differentiable functions. A solution to the \emph{$\kappa$-regular Markov local-field equation}, with potentials $(U, W)$ and initial law $\lambda$, is a tuple {that satisfies the following properties:}
    \begin{equation*}
        \Big((\Omega, \mF, \bbF, \bbP), (\mlfe, \gamma), (\mathbf{B}, \mathbf{X})\Big)
    \end{equation*}
    such that
    \begin{enumerate}
        \item $(\Omega, \mF, \bbP)$ is a probability space equipped with a filtration $\bbF = (\mF_t)_{t \geq 0}$;
        \item $\mathbf{B} := (B_v)_{v \in \{0, \ldots, \kappa\}}$ is a family of independent $1$-dimensional $\bbF$-Brownian motions;
        \item $\X := (X_v)_{v \in \{0, \ldots, \kappa\}}$ is a family of continuous $1$-dimensional  $\bbF$-adapted processes. Moreover, $\X$ is a Markov process with respect to its natural filtration;
        \item $\mlfe = \mL(\X)$, $\mlfe_0 = \lambda$ and $\mlfe \in \mPS_\kap^\infty$;
        \item $\gamma$ is a measurable map from $[0, \infty) \times \R^d\times \R^d$ to $\R^d$ that satisfies for all $t > 0$,
        \begin{equation}
            \gamma(t, X_0(t), X_1(t)) = \bbE\bigg[ U'(X_0(t)) + \sum_{v = 1}^\kap W'(X_0(t) - X_v(t)) \bigg| X_0(t), X_1(t) \bigg], \quad a.s.
            \label{e:gamma}
        \end{equation}
        \item $\X$ satisfies the following system of SDEs for $t \in (0, \infty)$:
        \begin{equation}
        \label{e:mlfe}
        \begin{split}
            dX_0(t) &= -\bigg(U'\big( X_0(t)\big) + \sum_{v = 1}^\kap W' \big( X_0(t) - X_v(t)\big)\bigg)dt + \sqrt{2} dB_0(t), \\
            dX_v(t) &= -\gamma(t, X_v(t), X_0(t)) dt + \sqrt{2} dB_v(t), \quad v = 1, \ldots, \kap.
        \end{split}
        \end{equation}
        \item For each $v \in \{1, \ldots, \kap\}$ and $T \in (0, \infty)$, we have
        \begin{align}
            \int_0^T \Big(\big| \gamma\big(t, X_0(t), X_v(t)\big) \big| ^2 + |\gamma\big(t, X_v(t), X_0(t)\big)\big|^2 \Big) dt < \infty \quad \text{a.s.}
            \label{e:mlfeInt}
        \end{align}
    \end{enumerate}
    \end{definition}{For ease of future reference, we define the drift $b$ as follows:}
                    \begin{align}
            b(\x) := U'(x_0) + \sum_{v = 1}^\kap  W'(x_0 - x_v), \quad \x \in \R^{1 + \kap}.
            \label{e:b}
        \end{align}
        When the underlying probability space and Brownian motion are clear, we denote a solution to the $\kap$-MLFE simply by $(\mlfe, \gamma, \X)$, or  simply $(\mlfe, \gamma)$ when we do not need explicit reference to the stochastic process $\X$. We also recall the definition of linear growth solution, which corresponds to Definition 3.9 of \cite{hu2024htheorem}.  

    \begin{definition}[Linear growth solution to $\kap$-MLFE]\label{d:linGrowth}
Let $(\mlfe, \gamma)$ be a solution to the $\kap$-MLFE on $[0, T]$ with coefficients $(\alpha, \beta)$ and initial condition $\lambda \in \mP(\R^{1 + \kap})$. We call $(\mlfe, \gamma)$ a \emph{linear growth solution} to the $\kap$-MLFE if
there exists $C \in (0, \infty)$ depending only on $( \kappa, d, T, b, \lambda)$ such that\begin{equation}
    \label{e:linGrowth}
        \sup_{t \in [0, T]} \operatorname*{\bar{\mlfe}_t-ess\,sup}_{x, y \in \R^d} \frac{|\gamma(t, x, y)|} {1 + |x| + |y|} \leq C.
    \end{equation}
\end{definition}

    \begin{definition}[Gaussian Markov local-field equation]\normalfont Given $\alpha, \beta \in \R$, define the potentials \begin{equation}
\label{e:potentials2}
    U(x) := \frac{\alpha + \beta}{2}|x|^2, \quad W(x) = -\frac{\beta}{2\kap}|x|^2, \quad x \in \R.
    \end{equation}
We refer to a solution to the $\kap$-MLFE with potentials $(U, W)$ as in \eqref{e:potentials2} a \emph{solution to the $\kap$-GMLFE with coefficients $(\alpha, \beta)$}.
\end{definition}

\subsubsection{Well-posedness of the $\kap$-GMLFE} In \cite{hu2024htheorem}, well-posedness of the general $\kappa$-MLFE was established under the condition that the gradient of the interaction potential $\nabla W$ {is} bounded, which clearly is not satisfied by \eqref{e:potentials2}. However, the results of long-time behavior in \cite{hu2024htheorem} hold even when $\nabla W$ is unbounded {as long as the $\kap$-MLFE is well-posed}, and we will leverage these results in the sequel to characterize the long-time behavior of the 2-GMLFE.
\begin{theorem}[Well-posedness of the $\kap$-GMLFE]
\label{t:r:wp}
    Let $\kap \in \bbN$ with $\kap \geq 2$ and $\alpha, \beta \in \R$. Suppose $\lambda \in \mPS_{\kap}$ is a non-degenerate centered Gaussian measure. Then there exists a linear growth solution $(\mlfe, \gamma)$ to the $\kap$-GMLFE on $[0, \infty)$ with coefficients $(\alpha, \beta)$ and initial condition $\lambda$ that is unique in the class of linear growth solutions. Moreover $\mlfe$ is the law of a Gaussian process such that the time marginals $\mu_t$ have law $\mathcal{N}(0, \Sigma_{\mu_t})$ {where $\{\Sigma_{\mu_t}\}_{t \geq 0}$ is a family of non-degenerate covariance matrices such that} $t \mapsto \Sigma_{\mlfe_t}$ is a continuous trajectory.
\end{theorem}
Theorem \ref{t:r:wp} is proved in Section \ref{ss:wpProof}, and is an immediate consequence of Proposition \ref{p:existence} and Proposition \ref{p:uniqueness}. The {proof of} existence follows by identifying a solution to the $\kap$-GMLFE through its covariance matrix, which solves a nonlinear matrix Riccati equation. Uniqueness is established through an entropy argument. 
\begin{remark}[A note about initial conditions]\normalfont \label{r:ic}
        We overcome the usual difficulties associated with studying conditional McKean-Vlasov equations with unbounded drift (see Section \ref{s:mot} for a discussion of such difficulties) by exploiting the Gaussian structure of our setting. However, this restricts our results to the case of Gaussian initial conditions, which means that Theorem \ref{t:r:wp} does not imply that the $\kap$-GMLFE is well-posed in the sense of Definition 3.9 of \cite{hu2024htheorem}, but rather only when restricted to the set of Gaussian initial conditions. Nevertheless, the well-posedness implied by Theorem \ref{t:r:wp} is sufficient for our results.
\end{remark}

\subsection{Gaussian local-field equation on the line.}

\subsubsection{Definition of the Gaussian local-field equation} In this section, we recall the definition of the LFE from Definition 3.5 of \cite{lacker2023marginal}, where the drift has the form given in \eqref{e:introLFE}. In the rest of the paper, we restrict ourselves to the case of $\kap =2$. 

\begin{definition}[Local-field equation on $\bbT_2$]
\label{d:g_lfe}
    Let $\alpha, \beta \in \R$ and $\lambda \in \mP(\R^3)$. A solution to the \emph{2-regular local-field equation} (henceforth abbreviated as \emph{2-LFE}) with  with potentials $(U, W)$ and initial law $\lambda$, is a tuple that satisfies the following properties \[\Big((\Omega, \mF, \bbF, \bbP), (\Gamma, \lfe), (\bfB, \bfZ)\Big)\] such that the following properties hold
    \begin{enumerate}
        \item $(\Omega, \mF, \bbP)$ is a probability space equipped with a filtration $\bbF = (\mF_t)_{t \geq 0}$;
        \item $\bfB = (B_v)_{v \in \{0, \pm 1\}}$ is a family of independent 1-dimensional $\bbF$-Brownian motions;
        \item $\bfZ = (Z_v)_{v \in \{0, \pm 1\}}$ is a family of continuous 1-dimensional $\bbF$-adapted processes;
        \item $\lfe$ satisfies $\lfe = \mathcal{L}(\bfZ)$ with $\lfe_0 = \lambda$ and $\lfe \in \mPS_2^\infty$;
        \item $\Gamma$ is a progressively measurable function from $\R_+ \times \mC \times \mC$ to $\R$ and satisfies for all $t > 0$, 
        \begin{equation}
    \Gamma(t,Z_0, Z_1) = \mathbb{E}\bigg[ U'\big( Z_0(t)\big) + \sum_{v \in\{ \pm 1\}} W'\big(Z_0(t) - Z_v(t)\big) \,\bigg|\, Z_0[t], Z_1[t]\bigg], \quad a.s.
    \label{e:Gamma}
\end{equation}
\item $\bfZ$ satisfies the following system of SDEs for $t \in (0, \infty)$:
\begin{equation}
\label{eq:model:localEquations}
    \begin{aligned}
    dZ_0(t) &= - \bigg( U'\big( Z_0(t)\big) + \sum_{v \in\{ \pm 1\}} W'\big(Z_0(t) - Z_v(t)\big)\bigg) dt + \sqrt{2} dB_0(t) ,\\
                dZ_{v}(t) &= -\Gamma\big(t, Z_{v}[t], Z_0[t]\big) dt + \sqrt{2} dB_v(t), \quad v \in \{\pm 1\}.
    \end{aligned}
\end{equation}
\item For each $i \in \{\pm 1\}$ and $T > 0$, we have
\begin{equation}
    \int_0^T \Big( |\Gamma(t,Z_0, Z_i)|^2 + |\Gamma(t,Z_i, Z_0)|^2 \Big)dt < \infty,\quad a.s. \label{e:lfe_int}
\end{equation}
\end{enumerate}
\end{definition}
    As in the case of the MLFE, when the underlying probability space and Brownian motion are clear, we denote a solution to the 2-LFE simply by $(\lfe, \Gamma, \bfZ)$, or  just $(\lfe, \Gamma)$ when we do not need to explicitly reference the stochastic process $\bfZ$.  

    \begin{definition}[Gaussian local-field equation]\normalfont\label{d:lfe} Given $\alpha, \beta \in \R$, let $(U, W)$ be as in \eqref{e:potentials}.
We refer to a solution to the 2-LFE with potentials $(U, W)$ as in \eqref{e:potentials} a \emph{solution to the 2-GLFE with coefficients $(\alpha, \beta)$}.
\end{definition}
\begin{remark}[Correspondence with \cite{lacker2023marginal}]\normalfont
Definition \ref{d:lfe} corresponds to Definition 3.5 of \cite{lacker2023marginal} in the particular case of \[d = 1, \quad \kap = 2, \quad b(\x) = \alpha x_0 + \frac{\beta}{2}(x_1 + x_{-1}),\quad \sigma = \sqrt{2},\] with the added condition in (4) that $\rho \in \mPS_2^\infty$. 
\end{remark}

\begin{proposition}[Well-posedness of 2-GLFE]
    Let $\alpha, \beta \in \R$ and $\sigma > 0$. There exists a  solution $\lfe$ to the 2-GLFE with coefficients $(\alpha, \beta)$ and initial condition $\lambda \sim \mN(0, \sigma \mathsf{I}_3)$ that is unique in law. Moreover, $\rho$ is the $\{0, \pm 1\}$-marginal law of \eqref{e:zInfinite} with $(U, W)$ as in \eqref{e:potentials} and is the law of a Gaussian process on $\R_+$ taking values in $\R^3$.
\end{proposition}
\begin{proof}
    The existence and uniqueness in law of the solution follows from Theorem 3.12 of \cite{lacker2023marginal}. By Theorem 3.12, $\lfe$ is the $\{-1, 0, 1\}$-marginal of an equation of the form of equation (1.4) in \cite{lacker2023marginal}, and thus lies in $\mPS_2^\infty$. By standard results (e.g., see Theorem 3.1 of \cite{lacker2023localweakconvergence}), the SDE in \eqref{e:zInfinite} with $(U, W)$ as in \eqref{e:potentials} has a unique strong solution $Z$, which is an (infinite-dimensional) Gaussian process. Since Theorem 3.12 of \cite{lacker2023marginal} shows that $\rho$ is the $\{0, \pm1\}$ marginal of the law of $Z$, $\rho$ is the law of a Gaussian process on $\mathbb{R}_+$ taking values in $\mathbb{R}^3$.    
\end{proof}

\subsection{Long-time behavior of the 2-GLFE and 2-GMLFE}

\subsubsection{The stationary distribution} We now present our results about the long-time behavior of the 2-GLFE and 2-GMLFE on the line. First, we introduce the common limiting distribution.

\begin{definition}[Limiting Gaussian distribution]
\label{d:r:stat} Let $\alpha, \beta \in \R$ with $\alpha > |\beta|$, and define the symmetric matrix
\begin{equation}
    \Sigma_\stat := \Mat_2(\sigma_0^\stat, \sigma_1^\stat, \sigma_2^\stat) = \begin{bmatrix}
        \sigma_0^\stat & \sigma_1^\stat & \sigma_1^\stat \\
        \sigma_1^\stat& \sigma_0^\stat &\sigma_2^\stat \\ 
        \sigma_1^\stat& \sigma_2^\stat& \sigma_0^\stat 
    \end{bmatrix},
    \label{e:SigmaStat}
\end{equation}
with
\begin{equation}
    \sigma^\stat_r := \frac{(-\beta)^r}{ \sqrt{\alpha^2 - \beta^2}\big(\alpha + \sqrt{\alpha^2 - \beta^2}\big)^r}, \quad r = 0, 1, 2.
    \label{e:sigma_pi}
\end{equation}Let $\stat$ denote the centered Gaussian measure $\mN(0, \Sigma_\stat)$.
\end{definition}
\begin{remark}[Properties of $\pi$] \normalfont
\label{r:piNonDegen}
    Let $\delta := \sqrt{\alpha^2 - \beta^2}$ and $\phi := -\beta/(\alpha + \delta)$. Then $\sigma_r^\stat = \phi^r/\delta$ and we see that $\sigma_0^\stat \sigma_2^\stat = (\sigma_1^\stat)^2$. Moreover, a simple calculation shows that the determinant of $\Sigma_\stat$ satisfies
    \begin{equation}
        \det\big(\Sigma_{\stat}\big) = (\sigma_0^\stat - \sigma_2^\stat)\big((\sigma_0^\stat)^2 - (\sigma_1^\stat)^2\big) =\frac{4}{\sqrt{\alpha^2 - \beta^2}(\alpha + \sqrt{\alpha^2 - \beta^2})^2}.
    \end{equation}
    Thus $\Sigma_\stat$ is non-degenerate when $\alpha > |\beta|$.  Moreover, since $\stat$ is a Gaussian measure, its conditional independence properties can be read off from its precision matrix $\Sigma_{\stat}^{-1}$. By Definition \ref{d:r:stat} and Definition \ref{d:mSpace}, for $i, j \in \{0, \pm 1\}$ we have
    \begin{equation}
        (\Sigma_\stat^{-1})_{i,j} = \frac{1}{|\Sigma_\stat|}\left \{ \begin{aligned}
            &(\sigma_0^\stat)^2 - (\sigma_2^\stat)^2, &\quad & i = j = 0, \\
            &(\sigma_0^\stat)^2 - (\sigma_1^\stat)^2, &\quad & i = j \neq 0, \\
            & \sigma_1^\stat(\sigma_2^\stat - \sigma_0^\stat), &\quad &|i - j| = 1, \\
            & (\sigma_1^\stat)^2 - \sigma_0^\stat\sigma_2^\stat, &\quad  &|i - j| = 2, \\
        \end{aligned}\right.
        \label{e:covStatInv}
    \end{equation}
    Thus $(\Sigma_\stat^{-1})_{1, -1} = (\Sigma_\stat^{-1})_{-1, 1} = 0$ and for $(Y_0, Y_1, Y_{-1}) \sim \stat$, $Y_{-1}$ and $Y_1$ are conditionally independent given $Y_0$. Therefore $\stat$ is a Markov random field (MRF), in the sense that it exhibits the following conditional independence property:
    \begin{equation}
    \label{e:pi-mrf}
        \stat(x_{-1}|x_0, x_1) = \bar{\stat}(x_{-1}|x_0)
    \end{equation}
\end{remark}

\subsubsection{Long-time behavior}

Our first result shows that the 2-GLFE converges exponentially fast to $\stat$. 

\begin{theorem}[Exponential convergence of the  2-GLFE]
\label{t:r:lfeConv}
    Let $\alpha, \beta \in \R$ with $\alpha > |\beta|$, and let $\stat$ be as in Definition \ref{d:r:stat}. Suppose $\lambda = \mN(0, \sigma \mathsf{I}_3)$, and let $\lfe$ be the unique solution to the  2-GLFE with coefficients $(\alpha, \beta)$ and initial condition $\lambda$, and let $\rho_t$ denote the $t$-marginal {of $\rho$}. Then there exist constants $C, c \in (0, \infty)$ that depend only on $(\alpha, \beta, \lambda)$ such that
    \begin{equation}
        d_{\text{TV}}(\lfe_t, \stat) \leq Ce^{-ct}, \quad t \geq 0.
        \label{e:lfeConv}
    \end{equation}
\end{theorem}

We prove Theorem \ref{t:r:lfeConv} in Section \ref{ss:ConvLFE} by using the Gaussian structure to obtain an explicit characterization of the solutions to the 2-GLFE.

Our next result shows that the 2-GMLFE also converges exponentially fast to $\stat$.

\begin{theorem}[Exponential convergence of the 2-GMLFE]
\label{t:r:mlfeConv}
    Let $\alpha, \beta \in \R$ with $\alpha > |\beta|$. Suppose $\lambda$ is a non-degenerate centered Gaussian measure in $\mPS_2$, and let $\mlfe$ be the unique solution to the 2-GMLFE on $[0, \infty)$ with coefficients $(\alpha, \beta)$ and initial condition $\lambda$. Then there exist constants $C, c \in (0, \infty)$ that depend only on $(\alpha, \beta, \lambda)$ such that
    \begin{equation}
        d_{\text{TV}}(\mlfe_t,\stat) \leq Ce^{-ct}, \quad t \geq 0.
    \label{e:convMLFE}
    \end{equation}
\end{theorem}

    Theorem \ref{t:r:mlfeConv} is proved in Section \ref{ss:LongMLFE} by utilizing the results about long-time behavior {of the} MLFE in \cite{hu2024htheorem}. In particular, we use properties of the \emph{sparse free energy} (see (4.2) in \cite{hu2024htheorem} or \eqref{e:sfe} below) and a corresponding \emph{H-theorem} (see Theorem 4.1 in \cite{hu2024htheorem} or Theorem \ref{t:ogHthm} below) to characterize the long-time behavior of the 2-GMLFE. 

The following corollary establishes Conjecture \ref{c:LFE_MLFE} for $\kap = 2$ and $(U, W)$ as in \eqref{e:potentials}. It is a direct consequence of Theorem \ref{t:r:lfeConv}, Theorem \ref{t:r:mlfeConv}, and the triangle inequality.
\begin{corollary}[2-GLFE and 2-GMLFE are exponentially close]
\label{c:r:lfeMlfe}
    Let $\alpha, \beta, \sigma \in \R$ with $\alpha > |\beta|$ and $\sigma > 0$. Suppose $\lambda = \mN(0, \sigma \mathsf{I}_3)$ and let $\lfe$ and $\mlfe$ be solutions to the 2-GLFE and 2-GMLFE, respectively, on $[0, \infty)$ with coefficients $(\alpha, \beta)$ and initial condition $\lambda$. Then there exist constants $C, c \in (0, \infty)$ depending only on $(\alpha, \beta, \sigma)$ such that
    \begin{equation}
        d_{\text{TV}}(\lfe_t, \mlfe_t) \leq C e^{-ct}.
    \end{equation}
\end{corollary}

\section{Long-Time Behavior of the 2-GLFE}

\subsection{Explicit solutions to the 2-GLFE}
\label{ss:SolLFE} The Gaussian structure in our setting facilitates the exact computation of the time-marginals of solutions to the $2$-GLFE. This is the content of the following proposition.

\begin{proposition}[2-GLFE solutions]
\label{p:lfeSol}
Let $\alpha, \beta \in \R$ and $\lambda = \mathcal{N}(0, \sigma \mathsf{I}_3)$. Let $\lfe$ be the solution to the 2-GLFE on $[0, \infty)$ with coefficients $\alpha, \beta$ and initial condition $\lambda$.  For all $t \in [0, \infty)$ we have $\lfe_t \sim \mathcal{N}(0, \Sigma_{\lfe_t})$ with\begin{equation}
    \Sigma_{\lfe_t} := \Mat_2\big(\sigma_0^{\lfe_t}, \sigma_1^{\lfe_t}, \sigma_2^{\lfe_t}\big),
\end{equation}
where $\Mat_2:\R^3 \rightarrow \R^{3 \times 3}$ is the map introduced in Definition \ref{d:mSpace}, and for $r = 0,1,2$, the entries $\sigma_r^{\rho_t}$ depend only on $\alpha, \beta, \sigma$ and are given by
\begin{equation}
    \sigma_r^{\lfe_t} := \big(-\text{\normalfont sign}\,\beta\big)^r\bigg[\sigma^2 e^{-2t\alpha} I_r(2t|\beta|) + 2 \int_0^t e^{-2s\alpha} I_r(2s|\beta|) ds\bigg],
    \label{e:sigma_r}
\end{equation}
with $I_r$ denoting the modified Bessel function of the first kind of order $r$, which is given by
\begin{equation}
\label{e:I_r}
    I_r(y) := \int_0^1 \cos(r \stat x) e^{y \cos(\stat x)}  dx.
\end{equation}
\end{proposition}
The proof of Proposition \ref{p:lfeSol} exploits results from \cite{lacker2023localweakconvergence} and \cite{lacker2023marginal} that show that the Gaussian 2-LFE is the limit of a sequence of Gaussian processes. More precisely, let $n \in \bbN$ be an even number and $\Cn = \{0, \ldots, n-1\}$ denote the cycle graph on $n$ elements with adjacency matrix $\sfA_n \in \R^{n \times n}$. For convenience, we make the identification $-1 := n-1$ in $\Cn$. Consider the continuous $n$-dimensional stochastic process $\bfZ^{n} = \{Z^{n}_v\}_{v \in \Cn}$ that satisfies the following SDE:
\begin{equation}
    d\bfZ^{n}(t) = - \Big(\alpha \sfI_n + \frac{\beta}{2} \sfA_n\Big) \bfZ^n(t) dt + \sqrt{2} d\bfB^n(t), \quad \bfZ^n(0) \sim \mathcal{N}(0, \sigma \sfI_n),
\label{e:nOU}
\end{equation}
where $\bfB^n = \{B^n_v\}_{v \in \Cn}$ is a collection of i.i.d. standard Brownian motions independent of the initial condition. The SDE \eqref{e:nOU} has a unique strong solution that is an Ornstein-Uhlenbeck process (e.g., see Section 5.6 and Problem 5.6.2 of \cite{karatzas1991book}). In other words $\bfZ^n$ is a centered Gaussian process with the explicit representation:
\begin{equation}
    \bfZ^n(t) = e^{-t\big(\alpha \sfI_n + \frac{\beta}{2} \sfA_n\big)} \bfZ^n(0) + \int_0^t e^{-(t-s)\big(\alpha \sfI_n + \frac{\beta}{2} \sfA_n\big)}\sqrt{2} d\bfB^n(s), \quad t \geq 0.
\end{equation}
Let $\bSigma_n(t) := \bbE[\bfZ^n(t) \otimes \bfZ^n(t)]$ denote the covariance matrix of $\bfZ^n(t)$. Then due to the independence of the $\bfZ^n(0)$ from $\bfB$, by the It\^{o} isometry we have
\begin{equation}
    \bSigma_n(t) = \sigma^2 e^{-2t\big(\alpha \sfI_n + \frac{\beta}{2} \sfA_n\big)} + 2 \int_0^t e^{-2s\big(\alpha \sfI_n + \frac{\beta}{2} \sfA_n\big)} ds.
    \label{e:bSgima}
\end{equation}

\begin{lemma}[$n$-particle covariance]\label{l:nCov} Let $n \in \bbN$ be even, and fix $\alpha, \beta \in \R$ and $\sigma \in (0, \infty)$. Let $\bfZ^n$ be the solution to \eqref{e:nOU} and let $\bSigma_n$ be its covariance matrix given by \eqref{e:bSgima}. Then for all $j, k \in \Cn$ we have
\begin{equation}
    \big(\bSigma_n(t)\big)_{j, k} = \frac{1}{n}\sum_{l = 0}^{n-1} \varphi(t, \lambda_l^n) \cos \bigg( \frac{2(j - k)\pi l}{n}\bigg),
    \label{e:bSig_n}
\end{equation}
where $\lambda_l^n := 2 \cos(2\pi l/n)$ for $l = 0, \ldots, n-1$ and
\begin{equation}
    \varphi(t, x) := \sigma^2 e^{-2\alpha t- \beta t x} + \frac{1 - e^{-2\alpha t - \beta tx}}{\alpha + \frac{\beta}{2}x}, \quad (t, x) \in [0, \infty) \times \R.
    \label{e:varphi}
\end{equation}
\end{lemma}

\begin{proof}
    As shown below, the lemma follows by direct calculation. Let $\omega_n := \exp(2 i \pi/n)$ denote the $n$-th root of unity. Given that the adjacency matrix $\sfA_n$ of the $n$-cycle $\Cn$ is a circulant matrix with zero on the diagonal, classical results about adjacency matrices of cycle graphs (e.g., see Section 1.4.3 of \cite{brouwer2012spectra}) show that $\sfA_n$ admits the spectral decomposition $\sfA_n = \sfS_n \sfD_n \sfS_n^*$ with 
    \begin{equation}
        (\sfS_n)_{j,k} = \frac{1}{\sqrt{n}}\omega_n^{jk}, \quad j, k \in \Cn.
        \label{e:sEntries}
    \end{equation} 
    For notational convenience, let $\bfs_k := \{ (\sfS_n)_{jk} \}_{j \in \Cn} \in \R^n$ denote the columns of $\sfS_n$. The matrix $\sfD_n$ is a diagonal matrix of eigenvalues given by $(\sfD_n)_{j,k} = \lambda_j^n \delta_{j,k}$  for $j, k \in \Cn$, where $\delta_{j, k}$ denotes the Kronecker delta and we have
    \begin{equation} 
        \lambda_j^n := 2 \cos \bigg( \frac{2\pi j}{n}\bigg), \quad j \in \Cn.
        \label{e:eigenvalues}
    \end{equation}
    Recall that $n$ is even. Notice that the extremal eigenvalues are $\lambda_0^n = 2$, $\lambda^n_{n/2} = -2$, and for $j \not \in \{0, n/2\}$ we have $\lambda_j^n = \lambda_{n - j}^n \not \in \{-2, 2\}$. The eigenspaces corresponding to $\lambda_0^n$ and $\lambda_{n/2}^n$ are one-dimensional, while the eigenspaces corresponding to $\lambda_j^n$ for $j \not \in \{0, n/2\}$ are two-dimensional.
    
    We can transform $\sfS_n$ to a more useful, real-valued orthonormal basis $\sfU_n  \in \R^{n \times n}$ with columns $(\mathbf{u}_k)_{k \in \Cn}$ given by
    \begin{equation}
    \label{e:uk}
        \mathbf{u}_k := \left \{ \begin{aligned}
            &\bfs_k, &\quad & k = 0, \frac{n}{2},\\
            &\frac{1}{\sqrt{2}}(\bfs_k + \bfs_{n - k}), &\quad &k = 1, \ldots, \frac{n}{2}-1,\\
            &\frac{1}{i\sqrt{2}}(\bfs_{n-k} - \bfs_k), &\quad &k = \frac{n}{2}+1, \ldots, n -1.
        \end{aligned} \right.\quad k \in \Cn.
    \end{equation}
    In particular, the orthonormal basis $\sfU_n$ is obtained by rotating each 2-dimensional eigenspace of $\sfA_n$ by a fixed amount. Therefore $\sfU_n$ spans the same eigenspaces as $\sfS_n$ and we have $ \sfA_n = \sfU_n \sfD_n \sfU_n^*$. We collect here a useful representation of the entries of $\sfU_n$, which is a consequence of \eqref{e:sEntries}-\eqref{e:uk}. For $j, k \in \Cn$, we have
    \begin{equation}
        (\sfU_n)_{j, k} = \delta_{0, k}\frac{1}{\sqrt{n}} + \delta_{\frac{n}{2}, k}\frac{(-1)^j}{\sqrt{n}} + \sum_{l = 1}^{\frac{n}{2}-1} \delta_{l, k} \sqrt{\frac{2}{n}}\cos \bigg( \frac{2\pi j k}{n}\bigg) - \sum_{l = \frac{n}{2}+1}^{n-1} \delta_{l, k} \sqrt{\frac{2}{n}}\sin \bigg( \frac{2\pi j k}{n}\bigg).
        \label{e:Uentries}
    \end{equation}
    By \eqref{e:bSgima}, \eqref{e:varphi}, and the fact that $\sfA_n = \sfU_n \sfD_n \sfU_n^*$ we have 
    \begin{equation}
    \label{e:bSigman}
        \bSigma_n(t) = \sfU_n \bPhi_n(t) \sfU_n^*,
    \end{equation}
    where we define $\bPhi_n(t)$ to be the diagonal matrix with entries given by,
    \begin{equation*}
        \big(\bPhi_n(t)\big)_{j, j} := \sigma^2 e^{-2t(\alpha + \frac{\beta}{2}\lambda_j^n)} + 2 \int_0^t e^{-2s(\alpha + \frac{\beta}{2}\lambda_j^n)} ds = \varphi(t, \lambda_j^n),\quad j \in \Cn.
    \end{equation*}
    Combining the above display with \eqref{e:bSigman} and \eqref{e:Uentries}, for $j, k \in \Cn$ we obtain
    \begin{equation*}
    \begin{aligned}        (\bSigma_n(t))_{j, k}  
= &\sum_{l = 0}^{n-1} \varphi(t, \lambda_{l}^n)  (\sfU_n)_{j, k} (\sfU_n)_{k, l} \\
        = &\frac{1}{n}\varphi(t, \lambda_0^n) + \frac{(-1)^{j+k}}{n} \varphi(t, \lambda_{n/2}^n) + \frac{2}{n}\sum_{l = 1}^{\frac{n}{2}-1} \varphi(t, \lambda_{l}^n) \cos\bigg(\frac{2\pi j l}{n}\bigg) \cos\bigg(\frac{2\pi k l}{n}\bigg) \\&\hspace{20em} + \frac{2}{n}\sum_{l = \frac{n}{2}+1}^{n-1} \varphi(t, \lambda_l^n)  \sin\bigg(\frac{2\pi j l}{n}\bigg) \sin\bigg(\frac{2\pi k l}{n}\bigg) \\
        = &\frac{1}{n}\varphi(t, \lambda_0^n) + \frac{(-1)^{j+k}}{n} \varphi(t, \lambda_{n/2}^n) \\ & \hspace{8 em} + \frac{2}{n}\sum_{l = 1}^{\frac{n}{2}-1}\varphi(t, \lambda_{l}^n)\bigg[\cos\bigg(\frac{2\pi j l}{n}\bigg) \cos\bigg(\frac{2\pi k l}{n}\bigg) + \sin\bigg(\frac{2\pi j l}{n}\bigg) \sin\bigg(\frac{2\pi k l}{n}\bigg) \bigg],
    \end{aligned}
    \end{equation*}
    where in the last equality we used the change of indices $l \mapsto n - l$ in the second sum of the penultimate line and the fact that the sine function is odd. Combining the last display with the well-known identity $\cos(x)\cos(y) + \sin(x) \sin(y) = \cos(x - y)$ then allows us to conclude that for all $j, k \in \Cn$ we have
    \begin{equation*}
    \begin{aligned}
(\bSigma_n(t))_{j, k} &=   \frac{1}{n}\varphi(t, \lambda_0^n) + \frac{(-1)^{j+k}}{n} \varphi(t, \lambda_{n/2}^n) + 2\bigg[\frac{1}{n}\sum_{l = 1}^{\frac{n}{2}-1}\varphi(t, \lambda_{l}^n)\cos \bigg(\frac{2\pi l }{n}(j - k)\bigg)\bigg] \\
&= \frac{1}{n} \sum_{l = 0}^{n-1} \varphi(t, \lambda_l^n) \cos\bigg(\frac{2\pi l}{n}(j - k)\bigg),
    \end{aligned}
    \end{equation*}
    where in the last equality we use the fact that $\cos(\pi(j-k)) = (-1)^{j-k} = (-1)^{j+k}$ and the change of indices $n \mapsto n - l$ in one copy of the sum in the penultimate line, alongside the fact that the cosine function is even.
\end{proof} 

\begin{remark}[Convergence of stationary distributions]\label{r:pi_n_to_pi}\normalfont Taking $t \rightarrow \infty$ in \eqref{e:bSig_n} and using the ergodicity of the finite-dimensional OU process \eqref{e:zOU}, we see that $\pi_n = \mathcal{N}(0, \mathbf{\Sigma}^n(\infty))$ for $\mathbf{\Sigma}^n(\infty) \in \R^{V_n \times V_n}$, where for $j, k \in V_n$ we have
\begin{equation*}
    \big(\mathbf{\Sigma}^n(\infty)\big)_{j, k} = \frac{1}{n}\sum_{l = 0}^{n-1}  \frac{1}{\alpha + \beta \cos(2\pi l/n)} \cos \bigg( \frac{2(j - k) \pi l}{n}\bigg).
\end{equation*}
Therefore, taking $n \rightarrow \infty$ in the above display shows that the $\{-1, 0, 1\}$-marginal of $\pi_n$ converges to the probability measure  $\stat$ specified in Definition \ref{d:r:stat}. 
\end{remark}

\begin{proof}[Proof of Proposition \ref{p:lfeSol}]
Let $\lfe^n := \mL(\{Z^n_v\}_{v \in \{0, \pm 1\}}) \in \mP(\mC^3)$ denote the law of the marginal of $\bfZ^{n}$ on the neighborhood $\{0, \pm 1\}$. 
By Theorem 3.3(ii) of \cite{lacker2023localweakconvergence} and Theorem 3.12 of \cite{lacker2023marginal}, it follows that $\lfe^n \rightarrow\lfe$ in the weak topology as $n \rightarrow \infty$ (see the discussion around Theorem 3.19 of \cite{lacker2023marginal} for a related argument). Note that \eqref{e:varphi} implies that $\varphi$ is continuous (in fact infinitely differentiable) on $[0, \infty) \times \R$ when $\alpha > |\beta|$. Therefore by the convergence of $\rho^n$ to $\rho$ as $n \rightarrow \infty$ and Lemma \ref{l:nCov}, for $i, j \in \{0, \pm 1\}$ with $r := |i -j|$ we have
\begin{equation}
    (\Sigma_{\lfe_t})_{i, j} = \lim_{n \rightarrow \infty} \big(\Sigma_{\lfe^n_t}\big)_{i,j} = \lim_{n \rightarrow \infty}\frac{1}{n}\sum_{l = 0}^{n-1} \varphi(t, \lambda_l^n) \cos \bigg( \frac{2r\pi l}{n}\bigg) = \int_0^1 \varphi\big(t, \cos(2\pi x)\big) \cos(2r\pi x) dx,
    \label{e:covnlimit}
\end{equation}
where the convergence of the Riemann sum follows from the fact that $\varphi$ is continuous (see Theorem 6.8 of \cite{rudin1976analysis}). By \eqref{e:covnlimit} and the change of variables $x \mapsto 2 - \frac{1}{2}x$ and the fact that the cosine function is even, we have
\begin{equation*}
   (\Sigma_{\lfe_t})_{i, j} = \int_0^1 \varphi\big(t, \cos(\pi x)\big) \cos(r\pi x) dx,
\end{equation*}
Together with \eqref{e:varphi}, \eqref{e:I_r}, and Fubini's theorem, the last display implies that
\begin{equation*}
\begin{aligned}
    (\Sigma_{\lfe_t})_{i, j} &= \sigma^2 e^{-2t\alpha} \int_0^1 \cos(r \pi x) e^{-2t\beta \cos(\pi x)} dx + \int_0^1 \cos(r \pi x) \frac{1 - e^{-2t\alpha - 2t\beta \cos(\pi x)}}{\alpha + \beta \cos(\pi x)} dx \\  
        &= \sigma^2 e^{-2t\alpha} I_r(-2t\beta) + 2 \int_0^1\cos(r\pi x)\bigg[\int_0^t  e^{-2s\alpha - 2s \beta \cos(2\pi x)} ds \bigg] dx. \\
    &= \sigma^2 e^{-2t\alpha} I_r(-2t\beta) + 2 \int_0^t e^{-2s\alpha} I_r(-2s\beta) ds. \\
    &= \big(-\text{sign}\,\beta\big)^r\bigg[\sigma^2 e^{-2t\alpha} I_r(2t|\beta|) + 2 \int_0^t e^{-2s\alpha} I_r(2s|\beta|) ds\bigg].
\end{aligned}
\end{equation*}
We then conclude by \eqref{e:sigma_r} that $(\Sigma_{\lfe_t})_{i, j} = \sigma_r^{\rho_t}$.
\end{proof}

{\begin{remark}\normalfont We note that the restriction to independent initial conditions in Proposition \ref{p:lfeSol} is not strictly necessary. A simple modification of the proof of Lemma \ref{l:nCov} shows that for any non-degenerate Gaussian initial measure $\lambda \in \mPS_2$, the first term of \eqref{e:bSgima} decays exponentially fast as $t \rightarrow \infty$. However, in this setting a closed form for $\bSigma_n$ may not be available.
\end{remark}}

\subsection{Convergence to equilibrium of the $2$-GLFE solution} {We now prove Theorem \ref{t:r:lfeConv}. Recall the matrix notation introduced in Section \ref{ss:vec-matrices}.}
\label{ss:ConvLFE}
\begin{proof}[Proof of Theorem \ref{t:r:lfeConv}]
    Fix $t > 0$. By Proposition \ref{p:lfeSol}, $\rho_t$ is a centered Gaussian distribution. Similarly, by Remark \ref{r:piNonDegen}  $\stat$ is a centered, non-degenerate Gaussian distribution. By the second assertion of Theorem 1.1 of \cite{devroye2018total} and sub-multiplicativity of the Frobenius norm, there exists $C_0 \in (0, \infty)$ depending on $\stat$ such that
    \begin{equation}
        d_{\text{TV}}(\stat, \rho_t) \leq \big\| \big(\Sigma_\stat\big)^{-1} \Sigma_{\lfe_t} - I\big \|_{F} \leq C_0 \big\| \Sigma_{\lfe_t} - \Sigma_\stat\big\|_F.
        \label{e:sigDiff1}
    \end{equation}  
    Recall the definitions of $\sigma^{\lfe_t}$ and $\sigma^\stat$ from \eqref{e:sigma_r} and \eqref{e:sigma_pi}, respectively. By the equivalence of the Frobenius and $l^\infty$ matrix norms, there exists $C_1 \in (0, \infty)$ such that
    \begin{equation}    
    \big\| \Sigma_{\lfe_t} - \Sigma_\stat\big\|_F \leq C_1 \sup_{r \in \{0, 1, 2\}} |\sigma_r^{\lfe_t} - \sigma^\stat_r|.
    \label{e:sigDiff2}
    \end{equation}
    Since $I_r$ denotes the modified Bessel function of the first kind of order $r$, by (2.15.3.1) of \cite{prudnikov1986vol2} it follows that
    \begin{equation}
        2\int_0^\infty e^{-2\alpha s}I_r(2s|\beta|)ds = \frac{|\beta|^r}{\sqrt{\alpha^2 - \beta^2}(\alpha + \sqrt{\alpha^2 - \beta^2})^r}.
    \end{equation}
    The last display, when combined with \eqref{e:sigma_r}, \eqref{e:sigma_pi} and the crude estimate $|I_r(y)| \leq e^{|y|}$, yields
    \begin{equation}
    \begin{aligned}
        |\sigma_r^{\lfe_t} - \sigma_r^\stat| &= \bigg| \sigma^2 e^{-2\alpha t} I_r(2t|\beta|) + 2 \int_t^\infty e^{-2\alpha s} I_r(2s|\beta|) ds\bigg| \\
        & \leq \sigma^2 e^{-2 (\alpha - |\beta|)t} + 2 \int_t^\infty e^{-2(\alpha - |\beta|)s} ds \\
        & \leq \big(\sigma^2 + 4(\alpha - |\beta|)^{-1}\big) e^{-2(\alpha - |\beta|)t}.
    \end{aligned}
    \end{equation}
    Substituting the above estimate into \eqref{e:sigDiff2} and \eqref{e:sigDiff1}, we obtain \eqref{e:lfeConv}.
\end{proof}

\section{Well-Posedness of the $\kap$-GMLFE}

\subsection{Auxiliary Lemmas} We begin by introducing three auxiliary lemmas that will be used in the proof of well-posedness. 
{First, we recall a standard result for comparing the laws of diffusion processes, which is stated in terms of the relative entropy functional that is defined by 
    \begin{equation}
        \mH\big( \nu\big| \tilde{\nu} \big) := \left\{
        \begin{aligned}
            &\int_\mX \log\bigg( \frac{ d\nu}{d\tilde{\nu}}(x) \bigg) \nu(dx) 
            &\quad& \mbox{if}\quad
            \nu \ll \tilde{\nu}, 
            \\
            &\infty &\quad& \mbox{otherwise, }
        \end{aligned}
        \right.
        \label{e:entropy-def}
    \end{equation}
for two measures $\nu, \tilde{\nu}$ on a Polish space $\mX$.}

\begin{lemma}[Lemma 3.5 and Remark 3.6 of \cite{conforti2023projected}]
\label{l:entropyEstimate}
Let $d \in \mathbb{N}$ and $\lambda_0 \in \mathcal{P}(\R^d)$. Suppose $b^1, b^2:[0, T] \times \R^d \rightarrow \R^d$ are measurable and satisfy the linear growth condition, that is for $i = 1,2$,
\begin{equation}
    \sup_{(t, x) \in [0, T] \times \R^d} |b^i(t, x)| \leq C(1 + |x|)
\end{equation}
for some $C \in (0, \infty)$. For $i = 1, 2$, suppose $(\Omega^i, \mathcal{F}^i, \mathbb{F}^i, \mathbb{P}^i)$ is a filtered probability space supporting an $m$-dimensional Brownian motion $W^i$ and a continuous $m$-dimensional $\mathbb{F}^i$-adapted process $X^i$ satisfying
\begin{equation*}
    dX^i(t) = b^i(t, X^i(t)) dt + \sqrt{2}dW^i(t), \quad X^i(0) \sim \lambda_0.
\end{equation*}
Then the relative entropy between the laws of $X^1$ and $X^2$ satisfies
\begin{equation*}
    \mathcal{H} \big( \mathcal{L}(X^1[T])\,|\, \mathcal{L}(X^2[T]) \big) = \frac{1}{2} \mathbb{E}^{\mathbb{P}^1} \Bigg[ \int_0^T |b^1(t, X^1(t)) - b^2(t, X^1(t))|^2 dt\Bigg].
\end{equation*}
\end{lemma}

Next, we present a version of the classical \emph{weighted Pinsker inequality}.

\begin{lemma}[Theorem 2.1(ii) of \cite{villani2005pinsker}]
    \label{l:pinsker}
    Let $\mathcal{X}$ be a Polish space, and $\nu$ and $\nu'$ be probability measures in $\mP(\mathcal{X})$. For any measurable, vector-valued function $f: \mathcal{X} \rightarrow \R^d$ and $\eps \in (0, \infty)$, we have,
    \begin{equation*}
        |\langle \nu - \nu', f \rangle|^2 \leq \frac{4}{\eps}\bigg(1 + \log \int_{\mathcal{X}}\exp\Big({\frac{\eps}{2}|f|^2}\Big) d\nu'\bigg) \mathcal{H}(\nu|\nu').
    \end{equation*}
\end{lemma}

We also recall an {easily derived} formula for the conditional mean and variance of Gaussian distributions, {which is included here only because we specialize it to symmetric Gaussian measures in $\mPS_2$}. Its proof is deferred to Appendix \ref{ap:pdEstimate_condGauss}.

\begin{lemma}[Conditional Gaussian distribution] \label{l:condGauss} Suppose $\lambda \in \mPS_2$ is a centered non-degenerate Gaussian measure. Then there exist $a, b, c \in \R$ such that $a > \max\{|b|, |c|\}$ and $\lambda \sim \mathcal{N}(0, \Mat_2(a, b, c))$. Moreover, for $\Y = (Y_{-1}, Y_0, Y_1) \sim \lambda$ we have
\begin{equation}
\begin{aligned}
    \bbE^{\lambda}\big[Y_{-1}|Y_0 = x, Y_1 = y\big] &= \bigg(\frac{ab - bc}{a^2-b^2}\bigg)x + \bigg( \frac{ac - b^2}{a^2 - b^2}\bigg)y, \\
    \text{\normalfont Var}^\lambda\big[Y_{-1}|Y_0 = x, Y_1 = y\big] &= a - \frac{ab^2 - 2b^2c + ac^2}{a^2 - b^2}. 
\end{aligned}
\label{e:gaussConditionals}
\end{equation}
In particular, the conditional variance under $\lambda$ of $Y_{-1}$ given $Y_0$ and $Y_1$ is a (deterministic) constant. 
\end{lemma}

We now use the above result to compute the integral on the right-hand side of Lemma \ref{l:pinsker} in the case of conditional Gaussian measures. This fact will be highly useful in the proof of uniqueness in Proposition \ref{p:uniqueness} and the proof of Theorem \ref{t:r:mlfeConv} in Section \ref{ss:LongMLFE}.

\begin{lemma}[Conditional exponential integrability]
\label{l:Lambda}
Let   $\lambda \in \mPS_2$ be a centered, non-degenerate Gaussian measure. Let $\Y = (Y_{-1}, Y_0, Y_1) \sim \lambda$ and $\eps \in (0, \infty)$, and define 
    \begin{equation}
        \Lambda_\eps^\lambda(x, y) := \frac{4}{\eps} \log \bbE^\lambda \bigg[ \exp \Big( \frac{\eps}{2}\big|Y_{-1} - \bbE^\lambda[Y_{-1}|Y_0, Y_1]\big|^2 \Big)  \bigg| Y_0 = x, Y_1 = y\bigg], \quad x, y \in \R.
        \label{e:wp:Lambda1}
    \end{equation}
    Let $\sigma_\lambda$ be the conditional variance under $\lambda$ of $Y_{-1}$ given $(Y_0, Y_1)$, which by Lemma \ref{l:condGauss} is a nonnegative constant. Then for all $\eps < \sigma_\lambda^{-1}$, $\Lambda_\eps^\lambda(x,y) \equiv \Lambda_\eps^\lambda$ is constant in $(x, y)$ and we have
    \begin{equation}
        \Lambda_\eps^\lambda = \frac{2}{\eps} \log\bigg( \frac{2\pi}{1 - \eps \sigma_\lambda}\bigg).
    \end{equation}
\end{lemma}
\begin{proof}
    Let $\lambda(x_{-1}|x_0, x_1) := \lambda(\x)/\bar{\lambda}(x_0, x_1)$ denote the conditional density of $\lambda$ and set
    \begin{equation}
        m_\lambda(x, y) := \bbE^\lambda[Y_{-1}|Y_0 = x, Y_1 = y]
    \end{equation}
    {to be} the conditional mean of $Y_{-1}$ under $\lambda$. We can rewrite \eqref{e:wp:Lambda1} as
    \begin{equation}
        \Lambda_\eps^\lambda(x, y) = \frac{4}{\eps} \log \int_\R \exp \bigg( \frac{\eps}{2}|z - m_\lambda(x, y)|^2\bigg) \lambda(z|x, y) dz.
        \label{e:wp:Lambda2}
    \end{equation}
    By Lemma \ref{l:condGauss} we have $\lambda(\cdot|x, y) \sim \mathcal{N}\big( m_\lambda(x, y), \sigma_\lambda)$ and thus
    \begin{equation}
    \lambda(z|x, y) =  \frac{1}{\sqrt{2\stat \sigma_\lambda}} \exp \bigg( -\frac{1}{2\sigma_\lambda}|z - m_\lambda(x, y)|^2 \bigg). 
    \end{equation}
    Combining the last two displays yields
    \begin{equation}
\int_\R \exp \bigg( \frac{\eps}{2}|z - m_\lambda(x, y)|^2\bigg) \lambda(z|x, y) dz = \frac{1}{\sqrt{2\stat \sigma_\lambda}} \int_\R \exp \bigg( \frac{\eps \sigma_\lambda - 1}{2\sigma_\lambda}|z - m_\lambda(x, y)|^2\bigg)dz.    \end{equation}
    Then for all $\eps < \sigma_\lambda^{-1}$ we have
    \begin{equation}
    \begin{aligned}
         \frac{1}{\sqrt{2\pi \sigma_\lambda}} \int_\R \exp \bigg( \frac{\eps \sigma_\lambda - 1}{2\sigma_\lambda}|z - m_\lambda(x, y)|^2\bigg)dz 
        &= \sqrt{\frac{2\pi}{1 - \eps \sigma_\lambda}}.
    \end{aligned}
    \end{equation}
Substituting the last two displays into \eqref{e:wp:Lambda2} concludes the proof.\end{proof}
\subsection{Proof of Theorem \ref{t:r:wp}}
\label{ss:wpProof}
Throughout Section \ref{ss:wpProof}, we fix $\kap \in \bbN$ with $\kap \geq 2$.
\subsubsection{Preliminaries and a nonlinear Riccati equation}  For convenience we use $\sfI := \sfI_{1 + \kap}$ to denote the $(1 + \kap) \times (1 + \kap)$ identity matrix. We first define several matrix functionals that correspond to certain conditional expectation operators.
\begin{definition}[Auxiliary matrices]
        Let $\sfR \in \R^{(1 +\kap) \times (1 + \kap)}$ denote the matrix of all $1$'s and define the matrices $\sfQ, \sfP \in \R^{(1 +\kap) \times (1 + \kap)}$ as follows: 
    \begin{equation}
        \begin{aligned}
            \sfQ_{ij} := \begin{cases}
                1 &\quad i = j = 0, \\
                0 &\quad \text{otherwise}.
            \end{cases}, \quad 
            \sfP_{ij} := \begin{cases}
                1 &\quad i = 0,  j \neq 0 , \\
                0 &\quad \text{otherwise}.
            \end{cases} \\
        \end{aligned}
        \label{e:auxMat}
    \end{equation}
    for $i, j \in \{0, \ldots, \kap\}$. Let $\sfJ := \sfP + \sfP^*$ where recall $\sfP^*$ is the transpose of $P$ and $\sfK := \sfR - \sfI - \sfJ$. Then the definitions of $\mSpace_\kap$ and $\Mat_\kap$ from Definition \ref{d:mSpace} imply that 
    \begin{equation}
        \Mat_\kap(a, b, c) = a \sfI + b \sfJ + c \sfK, \quad (a, b, c) \in \R^3.
        \label{e:Mat2}
    \end{equation}
    \label{d:aMatrices}
\end{definition}

    We now define several functionals on $\mSpace_\kap$.
\begin{definition}[Conditional expectation operators]
\label{d:CEop}
Define the functions $f:\R^3 \rightarrow \R$ and $g:\R^3 \rightarrow \R$ by
\begin{equation}
\label{e:fg}
    f(a, b, c) := \frac{b(a - c)}{a^2 - b^2}, \quad g(a, b, c) := \frac{ac - b^2}{a^2-b^2}.
\end{equation}
Fix $\alpha, \beta \in \R$ and define the functionals $\ta, \tb:\mSpace_\kap \rightarrow \R^3$ in the following way. For $\sfA \in \mSpace_\kap$, define $ (a, b, c) := \Mat_\kap^{-1}(\sfA)$ and set 
    \begin{equation}
    \begin{aligned}
        &\ta(\sfA) := \left \{ \begin{aligned} &\alpha + \frac{\beta(\kap - 1)}{\kap}f(a, b, c) &\quad &\text{if }|a| \neq |b|, \\ &0 &\quad & \text{otherwise.} \end{aligned} \right. \\  &\tb(\sfA) := \left \{ \begin{aligned} &\frac{\beta}{\kap} + \frac{\beta(\kap - 1)}{\kap}g(a, b, c), &\quad &\text{if } |a| \neq |b|, \\
        &0 &\quad & \text{otherwise.} \end{aligned}\right.
        \label{e:t}
    \end{aligned}
    \end{equation}
    Let $\sfQ, \sfP$ be as in Definition \ref{d:aMatrices}. Define the map $\sfL: \mSpace_\kap \rightarrow \R^{(1 + \kap) \times (1 + \kap)}$ by
    \begin{equation}
        \sfL(\sfA) := \alpha \sfQ + \ta(\sfA)(\sfI  - \sfQ) + \frac{\beta}{\kap} \sfP + \tb(\sfA)\sfP^*.
        \label{e:sL}
    \end{equation}
\end{definition}

\begin{remark}[Conditional expectations and Riccati equation] \normalfont
   We now elucidate the role of $\ta$ and $\tb$ in defining the necessary conditional expectations arising in the $\kap$-GMLFE. In light of Lemma \ref{l:gSym}, Lemma \ref{l:condGauss}, and \eqref{e:sym1}, given any centered non-degenerate Gaussian measure $\theta \in \mPS_\kap$, and $\Y = (Y_0, \ldots, Y_\kap) \sim \theta$, it follows that
\begin{equation}
    \bbE^\theta\bigg[\alpha Y_{0} + \frac{\beta}{\kap}\sum_{v = 1}^\kap Y_v \bigg|Y_0 = x, Y_1 = y\bigg] = \ta(\Sigma_\theta) x + \tb(\Sigma_\theta)y, \quad x, y \in \R.\label{e:tCE}
\end{equation}
We emphasize that by the symmetry of $\theta$ described in Definition \ref{def:not:symProb} and Definition \ref{d:mSpace}, it follows that there exist $a, b, c \in \R$ such that $\mL(Y_v, Y_0, Y_1) = \mathcal{N}(0, \Mat_2(a, b, c))$ for some $a, b, c \in \R$ for all $v \in \{2, \ldots, \kap\}$.

Given $\alpha, \beta \in \R$, and $\sfL$ as defined in \eqref{e:sL}, define the function $\sfF = \sfF_{\alpha, \beta}: \mSpace_\kap \rightarrow \R^{(1 + \kap) \times (1 + \kap)}$ by
\begin{equation}
    \label{e:sF}
    \sfF(\sfA) := 2\sfI - \sfL(\sfA) \sfA - \sfA\sfL(\sfA)^* ,\quad \sfA \in \mSpace_\kap,
\end{equation}
For $\Sigma_\lambda \in \mSpace_\kap^+$, consider the following nonlinear Riccati equation for $t \in [0, \infty)$:
\begin{equation}
\label{e:riccati}
    \frac{d}{dt} \Ric_t = \sfF(\sfV_t), \quad \Ric_0 = \Sigma_\lambda,
\end{equation}
with $\sfL$ as in \eqref{e:sL}. We will establish well-posedness of this nonlinear Riccati equation and show that its solutions can be used to construct solutions to the $\kap$-GMLFE with coefficients $(\alpha, \beta)$ and initial condition $\lambda \sim \mathcal{N}(0, \Sigma_\lambda)$. 
\end{remark}

We conclude this section with a lemma that shows that the functions $f$ and $g$ defined in \eqref{e:fg}, while potentially singular, behave nicely on the set $\Mat_\kap^{-1}(\mSpace_\kap^+)$. Its proof can be found in Appendix \ref{ap:pdEstimate_condGauss}.
\begin{lemma}[Estimates on $f, g$]
\label{l:pdEstimate}
Suppose $\kap \in \bbN$ satisfies $\kap \geq 2$ and let $\sfA \in \mSpace_\kap^+$. Then $(a, b, c) := \Mat_\kap^{-1}(\sfA)$ satisfies \begin{equation}
    \label{e:0gaussLemma}
a > \max\{|b|, |c|\}.    
\end{equation}
Moreover, the functions $f, g :\R^3 \rightarrow \R$ defined in \eqref{e:fg} satisfy
    \begin{equation}
    \label{e:matrixIneq2}
        \big| f(a, b, c) \big| \leq 2,
    \end{equation}
    and
    \begin{equation}
    \label{e:matrixIneq1}
        \big| g(a, b, c)\big| \leq 1.
    \end{equation}
\end{lemma}

\subsubsection{Solving the nonlinear Riccati equation}
To establish well-posedness of the Riccati equation \eqref{e:riccati}, we begin by showing that the functions $\sfL$ and $\sfF$ (defined in \eqref{e:sL} and \eqref{e:sF}, respectively) are suitably regular. 

\begin{lemma}[Regularity of Riccati functionals]
\label{l:sF_sL_estimates}
The functions $\ta, \tb:\mSpace_\kap \rightarrow \R$ and the function $\sfL:\mSpace_\kap \rightarrow \R^{(1 + \kap) \times (1 + \kap)}$, specified in Definition \ref{d:CEop} are uniformly bounded on $\mSpace_\kap^+$, that is 
    \begin{equation*}
        \sup_{\sfA \in \mSpace_\kap^+} \big\{|\ta(\sfA)| + |\tb(\sfA)| + \|\sfL(\sfA)\| \big\}< \infty.
    \end{equation*}Moreover, the function $\sfF:\mSpace_\kap \rightarrow \R^{(1 + \kap) \times (1 + \kap)}$ defined in \eqref{e:sF} maps $\mSpace_\kap$ into $\mSpace_\kap$, and the function $\sfF$ is uniformly Lipschitz on $\mSpace_\kap^+$,
    that is, there exists $C_F \in (0, \infty)$ such that
    \begin{align}
    \|\sfF(\sfA) - \sfF(\mathsf{B})\| & \leq C_F\|\sfA - \mathsf{B}\|, \quad \sfA, \mathsf{B} \in \mSpace_\kap^+. \label{e:sFLip}
    \end{align}
\end{lemma}

The proof of Lemma \ref{l:sF_sL_estimates} is deferred to Appendix \ref{ap:sF_sL_estimates}. Next, we establish well-posedness of the Riccati equation \eqref{e:riccati} by a Cauchy-Lipschitz argument. 

\begin{lemma}[Solution to the nonlinear Riccati equation] \label{l:riccati_wp} Suppose that $\alpha, \beta \in \R$ and $\Sigma_\lambda \in \mSpace_\kap^+$, and let $\sfF$ be the associated function defined in \eqref{e:sF}. Then there exists a unique global solution $\{\sfV_t\}_{t \geq 0}$ to the Riccati equation  \eqref{e:riccati}. Moreover, $\sfV_t \in \mSpace_\kap^+$ for all $t \in [0, \infty)$.
\end{lemma}

\begin{proof} By Lemma \ref{l:sF_sL_estimates}, the function $\sfF$ maps $\mSpace_\kap$ into $\mSpace_\kap$ and is uniformly Lipschitz on $\mSpace_\kap^+$. Thus, we view \eqref{e:riccati} as an ordinary differential equation (ODE) on the finite-dimensional Banach space $\mSpace_\kap$ (which is isomorphic to $\R^3$ through $\Mat_\kap$). Recall that $\mSpace_\kap^+$ is an open subset of $\mSpace_\kap$. By the Picard-Lindel\"{o}f theorem (e.g., see Theorem 7.6 of \cite{amann1990ODE}), there exists a unique maximal solution $\sfV:[0, T_*) \rightarrow \mSpace_\kap^+$ to \eqref{e:riccati} for some maximal existence time $T_* \in (0, \infty]$. Moreover, as $t \uparrow T_*$, $\sfV_t$ either explodes or reaches the boundary of $\mSpace_\kap^+$, that is
\begin{equation}
    \lim_{t \,\uparrow \,T_*} \min\Big\{ \inf_{\sfA \in \partial \mSpace_\kap^+}\|\sfV_t - \sfA\|, \, \|\sfV_t\|^{-1} \Big\} = 0.
    \label{e:pl_dichotomy}
\end{equation}

We now show that $T_* = \infty$. Suppose towards contradiction that $T_* < \infty$. Lemma \ref{l:sF_sL_estimates} implies that $\sfL$ is bounded on  $\mSpace_\kap^+$, which together with \eqref{e:sF} and \eqref{e:riccati} implies that there exists $C_0 \in (0, \infty)$ such that for all $t \in [0, T_*)$, we have \begin{equation*}
\begin{aligned}
\|\sfV_t\| \leq \|\sfV_0\| + \int_0^t \|\sfF(\sfV_s)\| ds  
     \leq  \|\sfV_0\| + 2\int_0^t\big( 1 +\| \sfL(\sfV_s)\| \|\sfV_s\| \big) ds \leq \|\sfV_0\| + C_0 \int_0^t \big(1 + \|\sfV_s\| \big) ds. 
\end{aligned}
\end{equation*}
By Gr\"onwall's inequality, this yields the estimate
\begin{equation*}   
\sup_{t \in [0, T_*)} \|\sfV_t\| \leq \big(C_0T_* + \|\sfV_0\|\big) e^{C_0T_*}.
\end{equation*}
Thus $\|\sfV_t\|$ is uniformly bounded for $t \in [0, T_*)$.
Therefore \eqref{e:pl_dichotomy} implies that $\sfV_t$ converges to the boundary of $\mSpace_\kap^+$ as $t \uparrow T_*$, that is
\begin{equation}
    \lim_{t\, \uparrow \,T_*}\Big\{ \inf_{\sfA \in \partial \mSpace_\kap^+}\|\sfV_t -\sfA\|\Big\} = 0.
    \label{e:PLconseq}
\end{equation}

We now show that $\sfV_t$ is bounded away from the boundary of $\mSpace_\kap^+$ for all $t \in [0, T_*)$. Since $\partial \mSpace_\kap^+ = \{ \sfA \in \mSpace_\kap: \det(\sfA) = 0\}$ is the set of matrices in $\mSpace_\kap$ whose determinant is zero, by the continuity of the determinant with respect to the Euclidean topology it will suffice to show that
\begin{equation}
    \sup_{t \in [0, T_*)} \frac{1}{\det(\sfV_t)} < \infty.
    \label{e:PDcontradiction}
\end{equation}
To this end, first note that by Jacobi's formula we have
\begin{equation*}
    \frac{d}{dt}\det(\sfV_t) = \det(\sfV_t)\text{tr}\bigg(\sfV_t^{-1} \frac{d}{dt}\sfV_t\bigg).
\end{equation*}
Since, as already argued above, $\sfV_t \in \mSpace_\kap^+$ is positive definite for $t \in [0, T_*)$, we have
\begin{equation*}
    \frac{d}{dt}\frac{1}{\det(\sfV_t)} = - \frac{1}{\det(\sfV_t)^2}\frac{d}{dt}\det(\sfV_t) = -\frac{1}{\det(\sfV_t)}\text{tr}\bigg(\sfV_t^{-1} \frac{d}{dt}\sfV_t\bigg), \quad t \in (0, T^*).
\end{equation*}
By \eqref{e:sF}, \eqref{e:riccati} and well-known cyclic properties of the trace, we have
\begin{equation*}
\begin{aligned}
\text{tr}\bigg(\sfV_t^{-1} \frac{d}{dt}\sfV_t\bigg) = \text{tr}\Big( 2 \sfV_t^{-1} - \sfV_t^{-1}\big(\sfL(\sfV_t)\sfV_t + \sfV_t \sfL(\sfV_t)^*\big)\Big) = 2 \text{tr}\big(\sfV_t^{-1}\big)    - 2 \text{tr}\big(\sfL(\sfV_t)\big).
\end{aligned}
\end{equation*}
The last two displays together imply
\begin{equation*}
    \frac{d}{dt}\frac{1}{\det(\sfV_t)} = \frac{2\text{tr}\big(\sfL(\sfV_t)\big)}{\det(\sfV_t)} - \frac{2\text{tr}\big(\sfV_t^{-1}\big)}{\det(\sfV_t)}.
\end{equation*}
As noted above $\sfV_t$ is positive definite for all $t \in [0, T_*)$, and thus we have $\text{tr}\big(\sfV_t^{-1}\big) \geq 0$ for all $t \in [0, T_*)$.
Recall the definition of $\ta$ from Definition \ref{d:CEop}. By the positive definiteness of $\sfV_t$ and Lemma \ref{l:sF_sL_estimates}  we have
\begin{equation*}
\begin{aligned}
    \sup_{t \in [0, T_*)}\big|\text{tr}\big(\sfL(\sfV_t)\big)\big| = \sup_{t \in [0, T_*)}\big|\alpha + (\kap - 1) \ta(\sfV_t)\big| < \infty.    
\end{aligned}
\end{equation*} Therefore there exists $C_1 \in (0, \infty)$ such that for all $t \in [0, T_*)$ we have
\begin{equation*}
\frac{d}{dt}\frac{1}{\det(\sfV_t)} \leq \frac{C_1}{\det(\sfV_t)}.
\end{equation*}
Applying Gr\"onwall's inequality to the previous display yields
\begin{equation*}
    \sup_{t \in [0, T_*)} \frac{1}{\det(\sfV_t)} \leq e^{C_1T_*} \frac{1}{|\sfV_0|}.
\end{equation*}
In particular, since it was assumed that $T_* < \infty$, we have 
\begin{equation*}
    \inf_{t \in [0, T_*)}\det(\sfV_t) \geq e^{-C_1T_*}|\sfV_0| > 0.
\end{equation*}
Thus \eqref{e:PDcontradiction} holds, which contradicts \eqref{e:PLconseq} and thus implies that $T_* = \infty$. 
\end{proof}

Finally, we show that the well-posedness of the Riccati equation implies the existence of a solution to the GMLFE.

\begin{proposition}[Existence of a solution to the $\kap$-GMLFE] 
\label{p:existence} Let $\alpha, \beta \in \R$ and suppose $\lambda \in \mPS_\kap$ is a centered non-degenerate Gaussian measure with covariance matrix $\Sigma_\lambda \in \mSpace_\kap^+$. Recall the definitions of $(\ta, \tb, \sfL)$ from Definition \ref{d:CEop} and the function $\sfF$ from \eqref{e:sF}. There is a unique solution $\{\Ric_t\}_{t \geq 0}$ to the Riccati equation \eqref{e:riccati}  with initial condition $\Ric_0 = \Sigma_\lambda$ such that $t \mapsto \Ric_t$ is continuous and $\Ric_t$ is positive definite for all $t \geq 0$. Furthermore, there exists a linear growth solution to the $\kap$-GMLFE $(\mlfe, \gamma, \X)$ with coefficients $(\alpha, \beta)$ and initial condition $\lambda$ such that $\X$ is a centered Gaussian process on $[0, \infty)$. Moreover, for every $t \geq 0$, $$\Sigma_{\mu_t} := \bbE[\X_t \otimes \X_t] = \Ric_t,$$ and $\X$ solves the SDE
\begin{equation}
    d\X(t) = - \sfL(\Ric_t) \X(t)+ \sqrt{2} d\bfB(t), \quad \X(0) \sim \lambda,
    \label{e:mlfeSol}
\end{equation}
and we have
\begin{equation}
    \gamma(t, x, y) := \ta(\Ric_t) x + \tb(\Ric_t) y, \quad (t, x, y) \in [0, \infty) \times \R \times \R.
    \label{e:gammaRic}
\end{equation}
\end{proposition}
\begin{proof} 
Lemma \ref{l:riccati_wp} ensures that there exists a (unique) solution $\{\sfV_t\}_{t\geq 0}$ to the Ricatti equation \eqref{e:riccati} and furthermore that $\sfV_t \in \mSpace_\kap^+$ for all $t \geq 0$. By Lemma \ref{l:sF_sL_estimates}, we have that $\sfL(\Ric_t)$ has entries that are uniformly bounded for $t \in [0, \infty)$. Hence by Theorem 5.2.9 and Proposition 5.2.13 of \cite{karatzas1991book}, there exists a unique strong solution to the SDE \eqref{e:mlfeSol}.

Moreover, since $\sfL(\sfV_t)$ is uniformly bounded for $t \in [0, \infty)$, by Problem 5.6.2 of \cite{karatzas1991book} we have that $\X = \{\X(t)\}_{t \geq 0}$ is a centered, continuous Gaussian process. For $t \in [0, \infty)$, let $\mlfe_t := \mL(\X(t))$. By Problem 5.6.1 of \cite{karatzas1991book}, we see that 
\begin{equation}
\frac{d}{dt} \Sigma_{\mlfe_t} = 2\sfI - \sfL(\Ric_t) \Sigma_{\mlfe_t} - \Sigma_{\mlfe_t} \sfL(\Ric_t)^*, \quad \Sigma_{\mlfe_0} = \Sigma_\lambda.
\label{e:sigRiccati}
\end{equation}
By the boundedness of $\sfL(\sfV_t)$, \eqref{e:sigRiccati} is a linear ODE on $\mSpace_\kap$ and by the Picard-Lindel\"of theorem, $\{\Sigma_{\mlfe_t}\}_{t \geq 0}$ is its unique solution. However, since $\{\Ric_t\}_{t \geq 0}$ satisfies \eqref{e:riccati}, by the definition of $\sfF$ in \eqref{e:sF}, $\{\sfV_t\}_{t \geq 0}$  is also a solution to \eqref{e:sigRiccati}. Thus we must have 
\begin{equation}
\Ric_t = \Sigma_{\mlfe_t}, \quad t \geq 0 .
\label{e:Ric=Sig}
\end{equation}

It only remains to show that $\X$ is a linear growth solution to the $\kap$-GMLFE. Let $\gamma$ be as in \eqref{e:gammaRic}, which by Lemma \ref{l:sF_sL_estimates} satisfies the linear growth condition \eqref{e:linGrowth}. Observe that \eqref{e:gammaRic}, \eqref{e:Ric=Sig}, \eqref{e:sL} of Remark \ref{r:conditonalDistributions}, and the fact that $\Sigma_{\mu_t} = \bbE\big[\X(t) \otimes \X(t)\big]$, together imply 
\[\gamma\big(t, X_0(t), X_1(t)\big) = \bbE\bigg[ \alpha X_0(t) + \frac{\beta}{\kap}\sum_{v =1}^\kap X_v(t) \bigg| X_0(t), X_1(t)\bigg], \quad t > 0. \]
This shows that \eqref{e:gamma} is satisfied with $(U, W)$ as defined in \eqref{e:potentials2}. Moreover, by \eqref{e:auxMat}, we have for all $\x = (x_0, x_1, \ldots, x_\kap) \in \R^{1 + \kap}$ that $(\sfQ \x)_0 = x_0 e_0$, $\sfP\x = \sum_{v = 1}^\kap x_ve_v$, and $\sfP^*\x = x_0\sum_{v= 1}^\kap e_v$, where $\{e_v\}_{v = 0}^\kap$ is the coordinate basis of $\R^{1 + \kap}$. Together with \eqref{e:sL} and \eqref{e:gammaRic}, this implies that for all $\x = (x_0, x_1, \ldots, x_\kap) \in \R^{1 + \kap}$
\begin{equation*}
    \Big(\sfL(\Ric_t)\x\Big)_v = \left \{ \begin{aligned} 
        &\alpha x_0 + \frac{\beta}{\kap}\sum_{v = 1}^\kap x_v, &\quad &v = 0, \\
        & \ta(\Ric_t) x_v + \tb(\Ric_t) x_0 = \gamma(t, x_v, x_0), &\quad & v \in \{1, \ldots, \kap\}.
    \end{aligned}\right.
\end{equation*}
Combining the above display with \eqref{e:mlfeSol} it follows that $\X(t)$ satisfies \eqref{e:mlfe}, again with $(U, W)$ given by \eqref{e:potentials2}. The integrability condition \eqref{e:mlfeInt} follows from the fact that $\gamma$ as defined in \eqref{e:gammaRic} is linear, and that $\X(t)$ is continuous almost surely. Thus we conclude that $(\mlfe, \gamma, \X)$ defines a linear growth solution to the $\kap$-GMLFE.
\end{proof}

\subsubsection{Uniqueness of the solution to the $\kap$-GMLFE}
\begin{proposition}[Uniqueness of a solution the $\kap$-GMLFE]
\label{p:uniqueness}  Let $\alpha, \beta \in \R$ and suppose $\lambda \in \mPS_\kap$ is a centered non-degenerate Gaussian measure. Then the tuple $(\mlfe, \gamma ,\X)$ described in Proposition \ref{p:existence} is the unique linear growth solution to the $\kap$-GMLFE with coefficients $(\alpha, \beta)$ and initial condition $\lambda$.
\end{proposition}

\begin{proof} Let $(\mlfe, \gamma, \X)$ be the solution to the $\kap$-GMLFE described in Proposition \ref{p:existence}. Since $\{\X_t\}_{t \geq 0}$ is a centered continuous Gaussian process, by Lemma \ref{l:gSym} and the fact that $\Mat_\kap$ is a linear isomorphism, the map $t \mapsto (a_t, b_t, c_t) := \Mat_\kap^{-1}(\Sigma_{\mlfe_t})$ is continuous. Suppose $(\mlfe^\dag, \gamma^\dag, \X^\dag)$ is another linear growth solution to the $\kap$-GMLFE. We may always assume that $(\X, \X^\dag)$ are constructed on the same filtered probability space $(\Omega, \mF, \bbF, \bbP)$. By Definition \ref{d:mlfe}, we have $\mlfe, \mlfe^\dag \in \mPS_\kap^\infty$. We write $\tilde{\mlfe}, \tilde{\mlfe}^\dag \in \mPS_2^\infty$ for the marginal distributions of $(X_0, X_1, X_2)$ and $(X^\dag_0, X^\dag_1, X^\dag_2)$  respectively. Define
    \begin{equation}
    \begin{aligned}
        m_t(x, y) &:= \int_{\R} z \tilde{\mlfe}_t(z|x, y) dz = \bbE[X_{2}(t)|X_0(t) = x, X_1(t) = y], \\
        m_t^\dag(x, y) &:= \int_{\R} z \tilde{\mlfe}^\dag_t(z|x, y) dz = \bbE[X^\dag_{2}(t)|X^\dag_0(t) = x, X^\dag_1(t) = y],
        \label{e:condMeans}
    \end{aligned}
    \end{equation}
    where we have used the notation introduced in Remark \ref{r:conditonalDistributions} for conditional densities in $\mPS_\kap$. {By Lemma \ref{l:condGauss} and the fact that $\tilde{\mlfe}_t \in \mPS_2$ is a non-degenerate Gaussian measure for $t\in  [0, \infty)$, it follows that the conditional variance of $\tilde{\mlfe}_t$, denoted  $\sigma^{\tilde{\mlfe}_t}$, is a deterministic function of $t$.} Note that this also implies that the covariance of $\tilde{\mu}_t$, denoted $\Sigma_{\tilde{\mu}_t}$ satisfies $\Sigma_{\tilde{\mu}_t} = \Mat_2(a_t, b_t, c_t)$ for all $t \geq 0$. We emphasize here that we do not assume that $\mu^\dag$ is a Gaussian measure. Since $\tilde{\mu}_t \in \mPS_2$ is non-degenerate, we can apply Lemma \ref{l:condGauss} and Lemma \ref{l:pdEstimate} to obtain
    \begin{equation} 
    \begin{aligned}
    |\sigma^{\tilde{\mlfe}_t}| &= \bigg|a_t - b_t \bigg(\frac{a_tb_t - b_tc_t}{a_t^2 - b_t^2} \bigg) - c_t\bigg(\frac{a_tc_t - b_t^2}{a_t^2 -b_t^2}\bigg)\bigg| \leq |a_t| + 2|b_t| + |c_t|, \quad t \in [0, \infty).
    \end{aligned}
    \label{e:condVarEstimateUnique}
    \end{equation}
    Since $(\mlfe, \gamma, \X)$ and $(\mlfe^\dag, \gamma^\dag, \X)$ are both linear growth solutions, we can apply Lemma \ref{l:entropyEstimate} and \eqref{e:sym1} to conclude that for all $T \in (0, \infty)$, {the relative entropy, see \eqref{e:entropy-def}, between the two laws takes the form}
    \begin{equation}
    \label{e:entropyUnique}
    \begin{aligned}
        \mH(\mlfe^\dag[T]|\mlfe[T]) &= \frac{1}{4}\sum_{v = 1}^\kap \int_0^T \bbE\Big[\big|\gamma^\dag(t, X^\dag_v(t), X^\dag_0(t)) - \gamma(t, X^\dag_v(t), X^\dag_0(t))\big|^2 \Big] dt \\
    &= \frac{\beta^2(\kap - 1)^2}{4\kap} \int_0^T\bbE\Big[\big|m_t^\dag(X_1^\dag(t), X^\dag_0(t)) - m_t(X^\dag_1(t), X^\dag_0(t))\big|^2\Big] dt.
    \end{aligned}
    \end{equation}
    Let $h(z; x, y) := z - m_t(x, y)$. By \eqref{e:condMeans}, for all $t > 0$ we have
    \begin{equation*}
    \begin{aligned}
        m_t^\dag(x, y) - m_t(x, y) &= \bbE\big[h(X^\dag_2(t);x, y)\big|X^\dag_0(t) = x, X^\dag_1(t) = y \big] - \bbE\big[h(X(t);x, y)\big|X_0(t) = x, X_1(t) = y \big] \\
        &= \langle h(\cdot;x, y) , \tilde{\mlfe}^\dag_t(\cdot|x, y) - \tilde{\mlfe}_t(\cdot|x, y) \rangle.
    \end{aligned}
    \end{equation*}
    The above display, together with Lemma \ref{l:pinsker} and Lemma \ref{l:Lambda} applied with $\lambda = \tilde{\mu}_t$, yields the following inequality for all $\eps \in (0, \infty)$:
    \begin{equation}
         |m_t^\dag(x, y) - m_t(x, y)|^2 \leq \big[4\eps^{-1} + \Lambda_\eps^{\tilde{\mu}_t}(x, y)\big] \mH\big(\tilde{\mu}^\dag(\cdot|x, y)\big| \tilde{\mu}(\cdot|x, y)\big), \quad t > 0.
         \label{e:pinskerUnique}
    \end{equation}
    By \eqref{e:condVarEstimateUnique} and the continuity of $t \mapsto (a_t, b_t, c_t)$, there exists $\eps_T > 0$ such that $\sup_{t \in [0, T]}|\sigma^{\tilde{\mu}_t}| \leq \eps_T^{-1}$. Thus Lemma \ref{l:Lambda} implies that $\Lambda_{\eps_T}^{\tilde{\mu}_t}(x, y) = \Lambda_{\eps_T}^{\tilde{\mu}_t}$ is constant for all $t \in [0, T]$. Together with \eqref{e:entropyUnique}, \eqref{e:pinskerUnique}, chain rule for relative entropy, and the data-processing inequality, this implies
    \begin{equation*}
        \mH(\mlfe^\dag[T]|\mlfe[T]) \leq \frac{\beta^2(\kap - 1)^2}{4\kap}\big(4 \eps_T^{-1} + \Lambda_{\eps_T}^{\tilde{\mlfe}}\big)  \int_0^T \mH(\mlfe^\dag[t]|\mlfe[t]) dt.
    \end{equation*}
    By Gr\"onwall's inequality, we have $\mH(\mlfe^\dag[T]|\mlfe[T])$. Since this holds for all $T > 0$, it follows that $\mlfe^\dag = \mlfe$.
\end{proof}
\begin{remark}[A note about linear growth solutions] \normalfont
Uniqueness within the class of linear growth solutions can likely be extended to the class of all solutions to the $\kap$-GMLFE. The requirement of linear growth in the proof of Proposition \ref{p:uniqueness} is enforced to ensure that Lemma \ref{l:entropyEstimate} can be applied (in fact, if we changed this to Lemma 4.4 of \cite{MR4595391}, then we immediately have uniqueness in the class of solutions with locally bounded drift). Stronger \emph{a priori} control on the conditional expectation would be enough to extend to uniqueness in the class of all solutions.
\end{remark}

\section{Long-Time Behavior of the 2-GMLFE}

\subsection{Sparse free energy and the H-theorem}
\label{ss:Htheorem} Recall that $\x = (x_{-1}, x_0, x_1) \in \R^{3}$ denotes a vector. Let $\nabla$ denote the gradient on $\R^{3}$ and let $\partial_{x_v}$ for $v \in \{0,\pm 1\}$ denote the partial derivative with respect to $x_v$.

We start by applying the general results from  \cite{hu2024htheorem} to the present context. In the remainder of this section, let $U, W : \R \rightarrow \R$ be continuously differentiable functions.
Let $\mPS_2^{\text{ac}}$ be the set of probability measures in $\mPS_2$ that are absolutely continuous with respect to Lebesgue measure on $\R^{1 + \kap}$. In what follows, we identify measures $\nu \in \mPS_2^{\text{ac}}$ with their densities.
For $\nu \in \mPS_2^{\text{ac}}$, define the \emph{sparse free energy} as
\begin{equation}
    \bbH_2(\nu) := \int_{\R^3} \Big( \log \nu(\x) -  \log \bar{\nu}(x_0, x_1) + U(x_0) + \frac{1}{2}\sum_{v \in \{\pm 1\}}K(x_0 - x_v)\Big) \nu( \mathbf{x})d\x, 
    \label{e:sfe}
\end{equation}
and the \emph{modified Fisher information} as follows:
\begin{equation}
    \bbI_2(\nu) :=
        \int_{\R^3} \bigg[ \Big| U'(x_0) + \sum_{v \in \{\pm1\}} K'(x_0 - x_v) + \partial_{x_0} \log \nu(\x)\Big|^2 + 2 \bigg|\partial_{x_1} \log \frac{\nu(\x)}{\bar{\nu}(x_0, x_1)}\bigg|^2 \bigg]\nu(\x)d\x.
    \label{e:mfi}
\end{equation}
The functionals $\bbH_2$ and $\bbI_2$ correspond to the quantities defined in (4.2) and (4.3), respectively, of \cite{hu2024htheorem} with $\kap = 2$ and $d = 1$.

We next show that the limiting distribution $\pi$ from Definition \ref{d:r:stat} for the LFE is a zero of the modified Fisher information.

\begin{lemma}[Zeros of the modified Fisher information]
\label{l:I(pi)=0}
    Suppose $\alpha, \beta \in \R$ satisfy $\alpha > |\beta|$. Let  $\stat \in \mPS_2^{\text{ac}}$ be as defined in Definition \ref{d:r:stat} and $(U, W)$ be as in \eqref{e:potentials}. Then $\stat$ is the unique zero of $\bbI_2$.
\end{lemma}

\begin{proof}   By \eqref{e:mfi} and \eqref{e:potentials}, we have that 
    \begin{equation}
        \bbI_2(\pi) = \int_{\R^3} \Big| \alpha x_0 + \frac{\beta}{2}(x_1 + x_{-1}) + \partial_{x_0} \log \stat(\x)\Big|^2\stat(\x)d\x + \int_{\R^3}2 \Big|\partial_{x_1} \log \frac{\stat(\x)}{\bar{\stat}( x_0, x_1)}\Big|^2 \stat(\x)d\x.
        \label{e:bbI=0}
    \end{equation}
    To prove the lemma, it suffices to show that the right-hand side of \eqref{e:bbI=0} is equal to zero.
    By {\eqref{e:pi-mrf} we have}
    \begin{equation}
\bigg|\partial_{x_1} \log\frac{\stat(\x)}{\bar{\pi}(x_0, x_1)} \bigg|^2 = \big|\partial_{x_1} \log\stat(x_{-1}|x_0, x_1)\big|^2 = 0.
\label{e:0statMFI}
    \end{equation}
        Next, we show that the first integral on the right-hand side of \eqref{e:bbI=0} is equal to zero. Let $\mathbf{w} = (w_{-1}, w_0, w_1)$ with $w_0 = \alpha$ and $w_i = \beta/2$ for $i \in \{\pm 1\}$, and define $b(\x) := \langle \mathbf{w}, \x \rangle$. We can rewrite the first integral in \eqref{e:bbI=0} as 
    \begin{equation}
    \begin{aligned}
        \int_{\R^3} \Big| b(\x) + \partial_{x_0} \log \stat(\x)\Big|^2 \stat(\x)d\x 
        = \int_{\R^3}\Big[  |b(\x)|^2 + 2 b(\x) \partial_{x_0} \log \stat(\x)
        + |\partial_{x_0} \log \stat(\x)|^2\Big] \pi(d\x).
    \end{aligned}
    \label{e:1statMFI}
    \end{equation}
    As usual let $\Y = (Y_{-1}, Y_0, Y_1) \sim \pi$. Since $\stat$ is a centered Gaussian measure, we have
    \begin{equation*}
        \int_{\R^3} |b(\x)|^2 \stat(\x)d\x = \bbE^\pi\big[ |\langle \mathbf{w}, \mathbf{Y} \rangle|^2] = \langle \mathbf{w}, \Sigma_\stat \mathbf{w} \rangle.
    \end{equation*}
    Recall  the helpful notation $\delta := \sqrt{\alpha^2 - \beta^2}$ and $\phi := - \beta/(\alpha + \delta)$ from Remark \ref{r:piNonDegen}. By \eqref{e:SigmaStat} we have
    \begin{equation*}
        \langle \mathbf{w}, \Sigma_\stat \mathbf{w} \rangle = \frac{1}{\delta}\bigg[ \alpha^2 + 2 \alpha \beta \phi + \frac{\beta^2}{2}(1 + \phi^2) \bigg] = \frac{1}{\delta}\bigg[ (\alpha + \beta \phi)^2 + \frac{\beta^2}{2}(1 - \phi^2) \bigg].
    \end{equation*}
    However, a simple calculation shows that
        $\alpha + \beta\phi = \delta$ and $1 - \phi^2 = \frac{2\delta}{\alpha + \delta},$
    and therefore
    \begin{equation*}
        \frac{1}{\delta}\bigg[ (\alpha + \beta \phi)^2 + \frac{\beta^2}{2}(1 - \phi^2) \bigg] = \delta + \frac{\beta^2}{\alpha + \delta} = \alpha.
    \end{equation*}
    The previous three displays together yield
    \begin{equation}
        \int_{\R^3} |b(\x)|^2 \stat(\x)d\x = \alpha.
        \label{e:11statMFI}
    \end{equation}
    Similarly, define $\mathbf{u} = (u_{-1}, u_0, u_1)$ where for $i \in \{0, \pm 1\}$ we have $$u_i = -(\Sigma_\stat^{-1})_{0i} = -\frac{1}{\det(\Sigma_\pi)}\left \{\begin{aligned}&(\sigma_0^\pi)^2 - (\sigma_2^\pi)^2, &\quad & i = 0, \\
    &\sigma^\pi_1(\sigma_2^\pi - \sigma_0^\pi), &\quad & i \in \{\pm 1\},
    \end{aligned} \right.$$ by \eqref{e:covStatInv}. Since $\mathbf{u}$ is the $0$-th column of $\Sigma_{\pi}^{-1}$, we clearly have $(\Sigma_\pi \mathbf{u})_i = -\delta_{i0}$ for $i \in \{0, \pm 1\}$ and thus
    \begin{equation}
        \int_{\R^3}b(\x) \partial_{x_0} \log \stat(\x) \stat(\x)d\x = \bbE^\stat \big[ \langle \mathbf{w}, \Y \rangle \langle \mathbf{u}, \Y \rangle] = \langle \mathbf{w}, \Sigma_\stat \mathbf{u} \rangle = -\alpha.
        \label{e:13statMFI}
    \end{equation}
Similarly, we have    
    \begin{equation*}
         \int_{\R^3} |\partial_{x_0} \log \stat(\x)|^2 \stat(\x)d\x = \bbE^\stat\big[|\langle \mathbf{u}, \Y \rangle|^2] = \langle \mathbf{u}, \Sigma_\pi \mathbf{u} \rangle = (\Sigma_\pi^{-1})_{00}.
    \end{equation*}
    Then \eqref{e:SigmaStat} implies that
    \begin{equation*}
        (\Sigma_\pi^{-1})_{00} = \frac{\sigma_{0}^\pi + \sigma_2^\pi}{(\sigma_0^\pi)^2 - (\sigma_1^\pi)^2} = \delta \frac{1 + \phi^2}{1 - \phi^2} = \alpha,
    \end{equation*}
    and by combining the previous two displays we have
    \begin{equation}
         \int_{\R^3} |\partial_{x_0} \log \stat(\x)|^2 \stat(\x)d\x =  \alpha.
        \label{e:12statMFI}
    \end{equation}
    Substituting \eqref{e:11statMFI}, \eqref{e:13statMFI}, and \eqref{e:12statMFI} into \eqref{e:1statMFI}, it follows that the right-hand side of \eqref{e:1statMFI} is zero. Together with \eqref{e:0statMFI} and \eqref{e:bbI=0}, this implies that $\bbI_2(\stat) =0$. \end{proof}

We now recall the key result of \cite{hu2024htheorem}, specialized to the case of the quadratic potentials \eqref{e:potentials}. 

\begin{theorem}[H-theorem for the 2-GMLFE] 
\label{t:ogHthm}
Let $\alpha, \beta \in \R$ satisfy $\alpha > |\beta|$. Suppose $U, W:\R \rightarrow \R$ are of the form \eqref{e:potentials}. Let $\lambda \in \mPS_2^{\text{\normalfont ac}}$ satisfy
\begin{equation}
    \bigg| \int_{\R} \lambda(x) \log \lambda(x) dx \bigg| + \int_{\R}|x|^2\lambda(x) dx < \infty.
    \label{e:lambdaEntVar}
\end{equation}
If $(\mu, \gamma)$ is the linear growth solution to the 2-GMLFE on $[0, \infty)$ with coefficients $(\alpha, \beta)$ and initial condition $\lambda$ (in the sense of Definition \ref{d:mlfe} and Definition \ref{d:linGrowth}), then the following identity holds for almost every $ 0 < r < t < \infty$:
\begin{equation*}
    \bbH_2(\mu_t) - \bbH_2(\mu_r) = - \int_r^t \bbI_2(\mu_s) ds,
\end{equation*}
where $\bbH_2$ and $\bbI_2$ are defined in \eqref{e:sfe} and \eqref{e:mfi}, respectively. Moreover, $\pi$ is the unique zero of $\bbI_2$ in $\mPS_2^{\text{\normalfont ac}}$, and there exist $c, C \in (0, \infty)$ such that we have
        \begin{equation}
            \mH(\mu_t|\pi) - \mH(\bar{\mu}_t|\bar{\pi}) = \bbH_2(\mu_t) - \bbH_2(\nu) \leq C e^{-ct}, \quad \text{ for a.e. }t \geq 0.
            \label{e:sfeDecay}
        \end{equation}
\end{theorem}

\begin{proof}
The first assertion follows from Theorem 4.1 of \cite{hu2024htheorem} on observing that Remark 4.6 of guarantees that $(U, W)$ satisfies Assumptions A, B$^*$, and C of \cite{hu2024htheorem}. Next, note that Lemma 8.5 of \cite{hu2024htheorem} implies that Assumption D of \cite{hu2024htheorem} is satisfied. Therefore by Theorem 4.17 of \cite{hu2024htheorem} there is a unique zero of $\mathbb{I}$, which by Lemma \ref{l:I(pi)=0} must be $\pi$. The last assertion \eqref{e:sfeDecay} follows from Theorem 4.18 of \cite{hu2024htheorem}.
\end{proof}
{The} last result of this section collects some important properties of $\pi$ that arise as a consequence of the results of \cite{hu2024htheorem}.
\begin{theorem}[Remark 4.9, Theorem 4.10, and Proposition 7.2 of \cite{hu2024htheorem}]
\label{t:ogPi}
Let $\alpha, \beta \in \R$ satisfy $\alpha > |\beta|$. Suppose $U, W:\R \rightarrow \R$ are of the form \eqref{e:potentials}. Then for $\pi$ as in Definition \ref{d:r:stat}, and $\Y = (Y_{-1}, Y_0, Y_1) \sim \pi$ there exists $C \in (0, \infty)$ such that
\begin{equation*}
\big|\bbE^\pi\big[W'(Y_0 - Y_v)| Y_0 = x]\big| \leq C(1 + |x|), \quad v \in \{\pm1\}, \quad  x \in \R.
\end{equation*}
Moreover, we have
\begin{equation}
    \partial_{x_v} \log \stat(\x) = \left \{ \begin{aligned}
        & - U'(x_0) - \sum_{v \in \pm 1} W'(x_0 - x_v), &\quad& v = 0, \\
        & - U'(x_v) -  W'(x_v - x_0) - \bbE^\stat[W'(Y_0 - Y_1)|Y_0 = x_v], &\quad& v \in \{\pm1\}.
    \end{aligned} \right.
    \label{e:mlfe:nablaPi}
\end{equation}
for all $\x \in \R^3$. In particular, $\nabla \log \pi(\x)$ satisfies a linear growth condition. \end{theorem}
\begin{proof}

By Lemma \ref{l:I(pi)=0}, we have $\bbI_2(\stat) = 0$. By Remark 4.6 of \cite{hu2024htheorem}, Assumption C of \cite{hu2024htheorem} is satisfied, and by Remark \ref{r:piNonDegen}, $\pi$ is a MRF. The first assertion follows from combining these two observations.

To show the second assertion, notice that by Remark 4.10 and Theorem 4.11 of \cite{hu2024htheorem} we have
    \begin{equation*}
        \stat(\x) = \exp \bigg( - U(x_0) - \sum_{v = \pm 1} \intPot(x_0 - x_v) \bigg) \prod_{v = \pm 1} e^{-\frac{1}{2}U(x_v)}\stat_0(x_v)^{ \frac{1}{2}}, \quad \x \in \R^3.
    \end{equation*}
Moreover, by Proposition 7.2 of \cite{hu2024htheorem}, we have
\begin{equation*}
     \frac{d}{dx}\big(\log \stat_0(x)\big) = - U'(x) - 2 \bbE^{\bar{\stat}}[W'(Y_0 - Y_1)|Y_0 = x], \quad x \in \R.
\end{equation*}
The claim \eqref{e:mlfe:nablaPi} follows immediately from the two previous displays.     
\end{proof}

\subsection{Exponential convergence of the 2-MLFE}\label{ss:LongMLFE}  The results of \cite{hu2024htheorem}, reproduced in Theorem \ref{t:ogHthm}, establish exponential decay of the sparse free energy; however, it is not immediately clear that this implies a corresponding rate of convergence of the 2-GMLFE in any other topology. We first show that the convergence result of Theorem \ref{t:ogHthm} can be strengthened in the present setting.

\begin{proposition}[Entropic convergence] \label{p:qual_entropic_conv} Fix $\alpha, \beta \in \R$ with $\alpha > |\beta|$ and let $\lambda$ be a non-degenerate centered Gaussian in $\mPS_2$. Also, let $\mlfe$ be the unique linear growth solution to the 2-GMLFE on $[0, \infty)$ with coefficients $(\alpha, \beta)$ and initial condition $\lambda$. 
Then we have 
\begin{equation}
    \lim_{t \rightarrow \infty} \mH(\mlfe_t|\stat) = 0.
    \label{e:entropyConv}
\end{equation}
\end{proposition}
\begin{proof}
    By Lemma \ref{l:I(pi)=0}, $\stat$ is the unique zero of $\bbI_2$. Then Theorem \ref{t:ogHthm} implies that $\mlfe_t \implies \stat$ as $t \rightarrow \infty$. By Definition \ref{d:r:stat} and Theorem \ref{t:r:wp}, $\stat$ and $\{\mlfe_t\}_{t \geq 0}$ are Gaussian measures, and therefore the weak convergence further implies that $\Sigma_{\mlfe_t} \rightarrow \Sigma_\stat$ entrywise as $t \rightarrow \infty$. Note that $\det(\Sigma_\stat) > 0$ by Remark \ref{r:piNonDegen}. Then by the continuity of the determinant, there exists $\eps > 0$ small enough and $t_\eps \in (0, \infty)$ such that we have $0 < \det(\Sigma_\stat) - \eps \leq  \det(\Sigma_{\mlfe_t}) \leq \det(\Sigma_\stat) + \eps$ for all $t > t_\eps$. Moreover by Theorem \ref{t:r:wp}, $\Sigma_{\mlfe_t}$ is continuous and non-degenerate on $[0, t_\eps]$ and hence $\det(\Sigma_t)$ is uniformly bounded away from $0$ and $\infty$ on $[0, t_\eps]$. Thus, $\inf_{t \geq 0} \det(\Sigma_{\mlfe_t}) > 0$.

    By the formula for the relative entropy of Gaussian measures (e.g., see Lemma A.5 of \cite{boursier2023cutoff}), we have
    \begin{equation*}
         \mH(\mlfe_t|\stat) = \frac{1}{2}\Big( \text{tr}(\Sigma_\stat^{-1}\Sigma_{\mlfe_t})  - 3 + \log \det(\Sigma_\stat) - \log \det(\Sigma_{\mlfe_t})\Big)
    \end{equation*}
    Since $\Sigma_{\mu_t} \rightarrow \Sigma_\pi$ as $t \rightarrow \infty$, by the continuity of matrix inversion, trace, and determinant with respect to the Euclidean topology (on any set of uniformly non-degenerate matrices) we have that the right-hand side of the above display converges to $0$ as $t \rightarrow \infty$. Thus, the advertised relative entropy convergence \eqref{e:entropyConv} holds.
\end{proof}

The stronger entropic convergence of Proposition \ref{p:qual_entropic_conv} is only qualitative and insufficient to obtain rates of convergence. In the following, we show how the sparse free energy can be used to strengthen further this result to exponential convergence in the total variation metric.
 
\begin{proof}[Proof of Theorem \ref{t:r:mlfeConv}]
The proof proceeds by controlling the total variation distance in the following way:\begin{equation}
    d_{\text{TV}}(\mlfe_t, \stat) \leq d_{\text{TV}}(\tilde{\stat}_{t/2}, \stat) +  d_{\text{TV}}(\tilde{\stat}_{t/2}, \mlfe_t),
    \label{e:triangleTV}
\end{equation}
where $\tilde{\stat}_{t/2}$ will be defined as the time-marginal of an appropriately chosen Langevin diffusion. The first term on the right-hand side of \eqref{e:triangleTV} will be estimated using standard results from the theory of Markov processes, while the second term will be controlled by the decay of the sparse free energy in \eqref{e:sfeDecay}.

Fix $t > 0$. First, we define $\tilde{\stat}_{t/2}$. Let $\tilde{\X}^\stat = (\tilde{X}^\stat_{1}, \tilde{X}^\stat_{0}, \tilde{X}^\stat_{-1})$ be the Langevin diffusion associated to $\stat$ that is given by
\begin{equation}
\begin{aligned}
    d\tilde{\X}_s^\stat &= \nabla \log \stat(\tilde{\X}_s^\stat) ds + \sqrt{2} d\bfB_s, \quad s > 0, \\
\end{aligned}
\label{e:pi_langevin}
\end{equation}
with initial condition $$   \mL(\tilde{\X}_0^\stat) = \mlfe_{t/2}.$$ By Theorem \ref{t:ogPi}, $\nabla \log \pi$ is of linear growth, which by Proposition 5.3.6 and 5.3.10 of \cite{karatzas1991book} implies that \eqref{e:pi_langevin} has a unique in law weak solution. We define the measure $$\tilde{\stat}_{t/2} := \mL(\tilde{\X}_{t/2}^\stat) \in \mP(\R^3).$$    

Next, we show exponential convergence of $\tilde{\stat}_{t/2}$ to $\stat$. Since $\stat$ is a centered non-degenerate Gaussian measure, Corollary 5.7.2 of \cite{bakry2014book} implies that $\stat$ satisfies a log-Sobolev inequality. Therefore by Theorem 5.2.1 of \cite{bakry2014book} there exists $c_0 \in (0, \infty)$ such that
\begin{equation*}
    \mH(\tilde{\stat}_{t/2}|\stat) \leq e^{-c_0t}\mH(\tilde{\stat}_0|\stat) =e^{-c_0t}\mH(\mlfe_{t/2}|\stat).
\end{equation*}
Moreover, \eqref{e:entropyConv} and the fact that $\lambda$ is a centered, non-degenerate Gaussian measure implies that $ \sup_{t \geq 0} \mH(\mlfe_t|\stat) < C_0$ for some $C_0 \in (0, \infty)$.
Pinsker's inequality then shows that 
\begin{equation}
    d_{\text{TV}}(\tilde{\stat}_{t/2}, \stat)^2 \leq \frac{1}{2} \mH(\tilde{\stat}_{t/2}| \stat) \leq C_0e^{-c_0t}.
    \label{e:decay1}
\end{equation}

Finally, we show that $\mlfe_t$ and $\tilde{\stat}_{t/2}$ are close in relative entropy. By Theorem \ref{t:r:wp}, $\mlfe_{t/2}$ is a centered non-degenerate Gaussian. Since the 2-GMLFE is a Markov process, by uniqueness of the solution to the 2-GMLFE (see Theorem \ref{t:r:wp}) it follows that $\mlfe_{s +t/2} = \tilde{\mlfe}_s$ for all $s \geq 0$, where $\tilde{\mlfe} = \{\tilde{\mlfe}_s\}_{s \geq 0}$ is the unique linear growth solution to the 2-GMLFE starting from $\mlfe_{t/2}$. Another application of Pinsker's inequality then shows that
\begin{equation}
    d_{\text{TV}}(\mlfe_{t}, \tilde{\stat}_{t/2})^2 = d_{\text{TV}}(\tilde{\mlfe}_{t/2}, \tilde{\stat}_{t/2})^2 \leq \frac{1}{2} \mH(\tilde{\mlfe}_{t/2} | \tilde{\stat}_{t/2}).
    \label{e:mlfePinsker1}
\end{equation}
As noted above, $\nabla \log \stat$ satisfies a linear growth condition. Recall the definition of the trajectorial laws $\tilde{\mu}[t/2]$ and $\tilde{\stat}[t/2]$ from Section \ref{s:notation}. Since $(\mlfe, \gamma, \X)$ is a linear growth solution, we can apply the data-processing inequality with Lemma \ref{l:entropyEstimate}, together with \eqref{e:mlfe:nablaPi}, \eqref{e:mlfe}, and the fact that $\mlfe \in \mPS_2^\infty$, to obtain the following estimate:
\begin{equation*}
\begin{aligned}
    \mH(\tilde{\mlfe}_{t/2} | \tilde{\stat}_{t/2}) &\leq \mH\Big(\, \tilde{\mlfe}\Big[\frac{t}{2}\Big] \,\Big|\, \tilde{\stat}\Big[\frac{t}{2}\Big] \,\Big) \\
    & \leq \frac{\beta^2}{8}\int_{t/2}^t \bbE\Big[\big|\bbE\big[X_{-1}(s)\big|X_0(s), X_1(s) \big] - \bbE^\stat\big[Y_{-1}\big|Y_0 = X_0(s)\big]\big|^2 \Big]ds,\\
    & = \frac{\beta^2}{8}\int_{t/2}^t \bbE\Big[\big|\bbE\big[X_{-1}(s)\big|X_0(s), X_1(s) \big] - \bbE^\stat\big[Y_{-1}\big|Y_0 = X_0(s), Y_1 = X_1(s)\big]\big|^2 \Big]ds,
\end{aligned}
\end{equation*}
where the last equality follows {from \eqref{e:pi-mrf}}. Let $m^\stat(x, y) := \bbE^\stat\big[Y_{-1}\big|Y_0 = x, Y_1 = y\big]$ and $h(\x) = x_{-1}$. By Lemma \ref{l:pinsker} and \eqref{e:wp:Lambda1}, for any $\eps > 0$  we have
\begin{equation*}
    \begin{aligned}
        \big|\bbE\big[X_{-1}(s)\big|X_0(s), X_1(s) \big] &- \bbE^\stat\big[Y_{-1}\big|Y_0 = X_0(s), Y_1 = X_1(s)\big]\big|^2 \\
        = &\big|\big\langle h(\cdot) - m^\stat\big(X_0(s), X_1(s)\big), \mlfe_s\big(\cdot|X_0(s), X_1(s)\big) - \stat\big(\cdot|X_0(s), X_1(s)\big) \big\rangle \big|^2 \\
        \leq & \Lambda_{\eps}^\stat\big(X_0(s), X_1(s)\big)  \mH\Big( \mlfe_s\big(\cdot|X_0(s), X_1(s)\big) \Big| \stat\big(\cdot|X_0(s), X_1(s)\big)\Big).
    \end{aligned}
\end{equation*}
By Lemma \ref{l:condGauss} and Definition \ref{d:r:stat}, the conditional variance of $\stat$ is bounded. Then by Lemma \ref{l:Lambda} there exists $\eps > 0$ such that $\Lambda_\eps^\stat(x, y)$ is constant and finite for all $x, y \in \R$. By the chain rule for relative entropy and the last two displays, we conclude there must exist a constant $\hat{C} \in (0, \infty)$ such that
\begin{equation*}
    \mH(\tilde{\mlfe}_{t/2} | \tilde{\stat}_{t/2}) \leq \hat{C} \int_{t/2}^t \big[\mH(\mlfe_s|\stat) - \mH(\bar{\mlfe}_s|\bar{\stat}) \big] ds.
\end{equation*}
Substituting \eqref{e:sfeDecay} into the above display and integrating, it is clear that there exist $c_1, C_1 \in (0, \infty)$ such that
\begin{equation*}
    \mH(\tilde{\mlfe}_{t/2} | \tilde{\stat}_{t/2}) \leq C_1e^{-ct}.
\end{equation*}
Together with \eqref{e:mlfePinsker1}, this implies that
\begin{equation}
       d_{\text{TV}}(\mlfe_{t}, \tilde{\stat}_{t/2})^2 \leq \frac{1}{2}\mH(\tilde{\mlfe}_{t/2} | \tilde{\stat}_{t/2}) \leq \frac{1}{2} C_1e^{-c_1t}.
        \label{e:decay2}
\end{equation}
We now combine \eqref{e:triangleTV}, \eqref{e:decay1}, and \eqref{e:decay2} to conclude the existence of $c, C \in (0, \infty)$ such that \eqref{e:convMLFE} holds.
\end{proof}

\section*{Acknowledgements}
K. Hu was supported by the National Science Foundation on the RTG grant NSF 
DMS-2038039 and K. Ramanan was supported by the National Science Foundation 
on the grant NSF DMS-2246838.  Both authors were also supported by the Office of Naval Research on the Vannevar Bush Faculty Fellowship grant ONR-N0014-21-1-2887, and
 would like to acknowledge support from the National Science Foundation under Grant No. DMS-1928930, since this paper was brought to completion when both authors were in residence at the Simons Laufer Mathematical Sciences Institute in Berkeley, California, during February 2025.
 
\appendix
\section{Proof of Lemma \ref{l:gSym}}\label{ap:ref_l:gsym}
\begin{proof}[Proof of Lemma \ref{l:gSym}]
    First, suppose that $\lambda \in \mPS_\kap$ is a centered Gaussian measure. We check that $\{0, \ldots, \kap\} \ni i \mapsto\Sigma^\lambda_{i,i}$ is constant. Since $\lambda$ satisfies the edge symmetry condition \eqref{e:sym2}, we have  \begin{equation*}
        \Sigma^\lambda_{i,i} = \bbE[X_i^2] = \bbE[X_j^2] = \Sigma^\lambda_{j,j}, \quad i, j \in \{0, \ldots, \kap\}.
    \end{equation*}
    Next, we check that $\Sigma^\lambda_{i,0} = \Sigma^\lambda_{0,i}$ is constant in $i \in \{1, \ldots, \kap\}$. Fix $i, j \in \{1, \ldots, \kap\}$ and pick a permutation $\tau$ on $\{1, \ldots, \kap\}$ such that $\tau(i) =j$. Then by the leaf exchangeability property \eqref{e:sym1} of $\lambda$, 
    \begin{equation*}
        \Sigma^\lambda_{i,0} = \bbE[X_iX_0] = \bbE[X_{\tau(i)} X_0] = \bbE[X_j, X_0] = \Sigma^\lambda_{j0}.
    \end{equation*}
    By \eqref{e:Mat}-\eqref{e:AltMat}, to conclude the first assertion it only remains to check that $\Sigma^\lambda_{i_1, j_1} = \Sigma^\lambda_{i_2, j_2}$ for all $(i_1, j_1)$ and $(i_2, j_2)$ in $\{(i, j): i \neq j, i \neq 0, j \neq 0\}$. We can find a permutation $\tau'$ on $\{1, \ldots, \kap\}$ such that $\tau'(i_1) = i_2$ and $\tau'(j_1) = j_2$. Then once again by leaf exchangeability \eqref{e:sym1}, 
    \begin{equation}
        \Sigma^\lambda_{i_1,j_1} = \bbE[X_{i_1}X_{j_1}] = \bbE[X_{\tau'(i_1)}X_{\tau'(j_1)}] = \bbE[X_{i_2}X_{j_2}] = \Sigma^\lambda_{i_2,j_2}.
    \end{equation}
    Therefore, we pick $a = \Sigma^\lambda_{0,0}$, $b = \Sigma^\lambda_{0,1}$, and $c = \Sigma^\lambda_{1,2}$ to see that $\Sigma^\lambda = \Mat_\kap(a, b, c)$ and conclude $\Sigma^\lambda \in \mSpace_\kap$.

    Now let $(a, b, c) \in \R^3$, suppose $\Sigma = \Mat_{\kap}(a, b, c)$ is positive semi-definite, and let $\lambda \in \mathcal{N}(0, \Sigma)$. Let $\tau$ be a permutation on $\{1, \ldots, \kap\}$. Then we have for all $i, j \in \{1, \ldots, \kap\}$ that $\Sigma_{i,i} = \Sigma_{\tau(i),\tau(i)} = a$, $\Sigma_{i,0} = \Sigma_{\tau(i),0} = b$, and $\Sigma_{i,j} = \Sigma_{\tau(i),\tau(j)} = c$. Therefore, $\lambda$ satisfies \eqref{e:sym1}. By \eqref{e:sym2}, it suffices to check that $\bar{\lambda}(x_0, x_1) = \bar{\lambda}(x_1, x_0)$. However, notice that $(X_0, X_1) \sim \mathcal{N}(0, \bar{\Sigma})$, where $\bar{\Sigma}_{00} = \bar{\Sigma}_{11} = a$ and $\bar{\Sigma}_{01} = \bar{\Sigma}_{10} = b$, which implies $(X_0, X_1) \stackrel{d}{=} (X_1, X_0)$. This shows $\lambda \in \mPS_\kap$.
\end{proof}

\section{Proof of Lemma \ref{l:pdEstimate} and Lemma \ref{l:condGauss}}
\label{ap:pdEstimate_condGauss}
\begin{proof}[Proof of Lemma \ref{l:pdEstimate}]
We collect a few facts about $\sfA  = \Mat_\kap(a, b, c)$. Since $\sfA$ is positive definite, the sub-matrices
\begin{equation*}
    R_1 := \begin{bmatrix}
        a & b \\
        b & a
    \end{bmatrix}, \,    R_2 := \begin{bmatrix}
        a & c \\
        c & a
    \end{bmatrix}, R_3 := \begin{bmatrix}
        a & b & b \\
        b & a & c \\
        b & c & a        
    \end{bmatrix}
\end{equation*}
are also positive definite and have positive determinants. This implies that $a^2 - b^2 = \det(R_1) > 0$ and $a^2 -c^2 = \det(R_2) > 0$, which together imply \eqref{e:0gaussLemma}. Moreover, the three eigenvalues of $R_3$ are
        $a - c$ and  $a + \frac{c}{2} \pm \frac{1}{2} \sqrt{8b^2 + c^2}$. Since $a - c > 0$ and $\det R_3 > 0$, it follows that the product of the last two eigenvalues is positive. In other words,\begin{equation}
    a^2 + ac - 2b^2 > 0.
        \label{e:keyPDineq}
\end{equation}
\indent We now show that both inequalities \eqref{e:matrixIneq1} and \eqref{e:matrixIneq2} are consequences of \eqref{e:0gaussLemma} and \eqref{e:keyPDineq} . We prove \eqref{e:matrixIneq1} first. Since $a^2 - b^2 > 0$ by \eqref{e:0gaussLemma}, by \eqref{e:keyPDineq} we have 
\begin{equation*}
    \frac{ac - b^2}{a^2 - b^2} > -1.
\end{equation*}
Additionally, since $a > |c|$ by \eqref{e:0gaussLemma}, we also have \begin{equation*}
    \frac{ac - b^2}{a^2 - b^2} \leq \frac{a|c| - b^2}{a^2 - b^2} < 1.
\end{equation*}
The last two displays together prove \eqref{e:matrixIneq1}. To prove \eqref{e:matrixIneq2}, first notice that \eqref{e:keyPDineq} is equivalent to \begin{equation}
    |b| \leq \sqrt{\frac{a^2 + ac}{2}}.
    \label{e:1gaussLemma}
\end{equation}
Next, we claim that \eqref{e:matrixIneq2} is implied by \begin{equation}
|b| \leq \frac{\sqrt{16a^2 + (a-c)^2} - (a-c)}{4},
\label{e:2gaussLemma}
\end{equation}
and hence, also by
\begin{equation}
\label{e:3gaussLemma}
    b^2 \leq \frac{1}{16}\Big( 16 a^2 + 2(a - c)^2 - 2(a-c) \sqrt{16a^2 + (a -c)^2}\Big).
\end{equation}
To see why the claim holds, note that since $a > c$, \eqref{e:2gaussLemma} implies 
\begin{equation}
\label{e:4gaussLemma}
    |b|(a-c) \leq \frac{1}{4}\Big( (a-c)\sqrt{16a^2 + (a-c)^2} - (a-c)^2\Big).
\end{equation}
When combined with \eqref{e:3gaussLemma}, this yields
\begin{equation*}
    2b^2 + |b|(a-c) \leq 2a^2,
\end{equation*}
which when rearranged yields \eqref{e:matrixIneq2}.
Let $\omega := a - c > 0$. We see from \eqref{e:1gaussLemma} and the fact that \eqref{e:3gaussLemma} implies \eqref{e:matrixIneq2}, it suffices to show  that
\begin{equation}
\label{e:keyPDineq2}
    a^2 - \frac{a\omega}{2} = \frac{a^2 + ac}{2} \leq \frac{16a^2 + 2\omega^2 - 2\omega\sqrt{16a^2 + \omega^2}}{16}.
\end{equation}
Since $\omega > 0$, \eqref{e:keyPDineq2} is equivalent to $\sqrt{16a^2 + \omega^2} \leq \omega + 4a$, which holds trivially since $a, \omega > 0$. This proves \eqref{e:keyPDineq2} and hence \eqref{e:matrixIneq2}.
\end{proof}

\begin{proof}[Proof of Lemma \ref{l:condGauss}]
    By Lemma \ref{l:gSym}, there exists $(a, b, c) \in \R$ with $a > 0$ such that $\lambda \sim \mathcal{N}(0, \Mat_2(a, b, c))$. Moreover, since $\lambda$ is non-degenerate $\Mat_2(a, b, c)$ is positive definite and we can apply Lemma \ref{l:pdEstimate} with $\kap = 2$ to deduce that $a > \max\{|b|, |c|\}$. Then \eqref{e:gaussConditionals} follows from the definition of $\Mat_2$ in \eqref{e:Mat} and well-known formulas for the conditional expectation and variance of Gaussian measures (e.g., see Appendix A.2 of \cite{rasmussen2006gaussian}). 
\end{proof}

\section{Proof of Lemma \ref{l:sF_sL_estimates}}
\label{ap:sF_sL_estimates}
\begin{proof}[Proof of Lemma \ref{l:sF_sL_estimates}]
Lemma \ref{l:pdEstimate}, \eqref{e:fg}, and \eqref{e:t} together imply that $\ta$ and $\tb$ are uniformly bounded on $\mSpace_\kap^+$. By \eqref{e:sL}, the boundedness of $\sfL$ on $\mSpace_\kap^+$ follows directly from the boundedness of $\ta$ and $\tb$.

Next, we show that $\sfF$ maps $\mSpace_\kap$ to $\mSpace_\kap$. This follows from computing the following matrix product:
    \begin{equation*}
2 \sfI - \sfL(\sfA) \sfA - \sfA \sfL(\sfA)^*, \quad \sfA \in \mSpace_\kap.
    \end{equation*}
     Let $\sfA \in \mSpace_\kap$ and $(a, b, c) := \Mat_\kap^{-1}(\sfA)$. By \eqref{e:Mat2}, we have
\begin{equation}
\sfA = a \sfI + b \sfJ + c \sfK.
    \label{e:Rictr}
\end{equation}
Recall the definitions of $\sfQ, \sfP, \sfJ$, and $\sfK$ from Definition \ref{d:aMatrices}. We note here that we have the following relations:
    \begin{equation}
    \label{e:mRelations}
    \begin{aligned}
        &\sfQ \sfJ = \sfJ (\sfI  - \sfQ) = \sfP, &\quad &\sfQ \sfK = \sfK\sfQ = \sfP^* \sfK = \sfK\sfP =0,  &\quad &\sfJ\sfP = \sfP^* \sfJ = \sfK + (\sfI  - \sfQ), \\
        &\sfP \sfJ = \sfJ\sfP^* = \kappa \sfQ, &\quad &(\sfP\sfK)^* = \sfK\sfP^* = (\kap - 1)\sfP^*,  &\quad &(\sfI  - \sfQ) \sfK = \sfK(\sfI  - \sfQ) = \sfK.        
    \end{aligned}
    \end{equation}
Full verification of \eqref{e:mRelations} can be found below. Combining the relations \eqref{e:mRelations} with \eqref{e:sL} yields
    \begin{equation*}
    \begin{aligned}
        \sfL(\sfA)\sfJ &= \beta \sfQ + \frac{\beta}{\kap}(\sfI  - \sfQ) +  \tilde{\alpha}(\sfA) \sfP + \alpha \sfP^* + \frac{\beta}{\kap} \sfK, \\
        \sfL(\sfA)\sfK &= \frac{\beta(\kap - 1)}{\kap} \sfP + \tilde{\alpha}(\sfA) \sfK.
    \end{aligned}
    \end{equation*}
    Then the above display, together with \eqref{e:Rictr} and \eqref{e:sL}, implies that
    \begin{equation*}
        \begin{aligned}
        \sfL(\sfA) \sfA +  \sfA\sfL(\sfA)^* =  2(\alpha a  &+ \beta b ) \sfQ + 2 \big(  \tilde{\alpha}(\sfA) a  +  \tilde{\beta}(\sfA) b  \big) (\sfI  - \sfQ) \\ &+ \bigg( \Big( \tfrac{\beta}{\kap} + \tilde{\beta}(\sfA) \Big)a 
        + \big(\alpha + \tilde{\alpha}(\sfA)\big)b  +  \frac{\beta (\kap - 1)}{\kap}c \bigg) \sfJ            \\
        &+ \big( \tilde{\beta}(\sfA)b  +  \tilde{\alpha}(\sfA)c \big) \sfK.
        \end{aligned}
    \end{equation*}
Furthermore, note that by \eqref{e:t} and Lemma \ref{l:pdEstimate} we have for any positive definite $A = \Mat_\kap(a, b, c)$ that
\begin{equation*}
\begin{aligned}
(\ta(\sfA) - \alpha)a + (\tb(\sfA) - \beta)b = \frac{\beta(\kap - 1) }{\kap}\bigg(\frac{a^2b - a b c}{a^2- b^2}\bigg) - \frac{\beta(\kap - 1)}{\kap}b + \frac{\beta(\kap-1)}{\kap}\bigg(\frac{abc - b^3}{a^2 - b^2}\bigg) = 0.
\end{aligned}
\end{equation*}
We then combine the previous two displays to obtain
        \begin{equation}
        \begin{aligned}
\sfF(\sfA) =2\sfI - \sfL(\sfA) \sfA  -  \sfA\sfL(\sfA)^* = \ff_1(\sfA) \sfI + \ff_2(\sfA)J + \ff_3(\sfA) \sfK,
        \end{aligned}
\label{e:sF_final}        
    \end{equation}
where
    \begin{equation}
        \begin{aligned}
            \ff_1(\sfA)&= 2(1 - \alpha a - \beta b), \\
            \ff_2(\sfA)&= -\bigg[\bigg(\frac{\beta}{\kap} + \tilde{\beta}(\sfA)\bigg)a + \big(\alpha + \tilde{\alpha}(\sfA)\big)b + \frac{\beta(\kap - 1)}{\kap}c\bigg], \\
            \ff_3(\sfA) &= - \tilde{\beta}(\sfA) b - \tilde{\alpha}(\sfA)c.
        \end{aligned}
        \label{e:ffi}
    \end{equation}
Then by \eqref{e:Mat2} and \eqref{e:sF_final}, we have that $\sfF(\sfA) \in \mSpace_\kap$.

Finally, we establish that $\sfF$ is uniformly 
Lipschitz on $\mSpace_\kap^+$. By \eqref{e:sF_final} and \eqref{e:ffi}, it will suffice to establish that $\ff_i:\mSpace_\kap^+ \rightarrow \R$ are uniformly Lipschitz for each $i = 1, 2, 3$. Since $\Mat_\kap : \R^3 \rightarrow \mSpace_\kap$ is a linear isomorphism, it will suffice to show that $\ff_i \circ \Mat_\kap^{-1} : \R^3 \rightarrow \R$ is uniformly Lipschitz for each $i = 1, 2, 3$. In the following, we abuse notation and simply write $\ff_i$ for $\ff_i \circ \Mat_\kap^{-1}$. 

Clearly, $\ff_1$ is a Lipschitz function. We then check that $\ff_2$ and $\ff_3$ are Lipschitz. By Definition \ref{d:CEop},  we have
\begin{equation}
\label{e:ffi_2}
\begin{aligned}
        \ff_2(a, b, c) &= - \bigg[ \bigg( \frac{2\beta}{\kap} + \frac{\beta(\kap - 1)}{\kap}g(a, b, c)\bigg)a + \bigg(2\alpha + \frac{\beta(\kap - 1)}{\kap}f(a, b, c)\bigg)b + \frac{\beta(\kap - 1)}{\kap} c\bigg] \\
        &= - \bigg[\frac{2 \beta}{\kap}a + 2\alpha b + \frac{\beta(\kap - 1)}{\kap}\big( a g(a, b, c) + b f(a, b, c) + c\big) \bigg]\\
        \ff_3(a, b, c) &= -\bigg[\bigg( \frac{\beta}{\kap} + \frac{\beta(\kap - 1)}{\kap}g(a, b, c)\bigg)b + \bigg(\alpha + \frac{\beta(\kap - 1)}{\kap}f(a, b, c)\bigg) c  \bigg] \\
        &= - \bigg[ \frac{\beta}{\kap}b + \alpha c + \frac{\beta(\kap - 1)}{\kap}\Big(bg(a, b, c) + c f(a, b, c)\Big)\bigg].
\end{aligned}
\end{equation}
Next, we compute $\nabla f(a, b, c)$ and $\nabla g(a, b, c)$. We have
\begin{equation}
\label{e:df}
\begin{aligned}
    \partial_af(a, b, c) &= \frac{b}{a^2-b^2}g(a, b, c) - \frac{a}{a^2 -b^2}f(a, b, c), \\
    \partial_bf(a, b, c) &= \frac{a^2 + b^2}{b(a^2-b^2)}f(a, b, c), \\ 
    \partial_cf(a, b, c) &= - \frac{b}{a^2-b^2}.
\end{aligned}
\end{equation}
Similarly, we have
\begin{equation}
\label{e:dg}
    \begin{aligned}
        \partial_a g(a, b, c) &= \frac{b}{a^2-b^2}f(a, b, c) - \frac{a}{a^2-b^2}g(a, b, c), \\
        \partial_bg(a, b, c) &= - \frac{2a}{a^2 - b^2}f(a, b, c), \\
        \partial_c g(a, b, c) &= \frac{a}{a^2 - b^2}.
    \end{aligned}
\end{equation}
By Lemma \ref{l:pdEstimate}, the functions $f, g$ are bounded on $\Mat_\kap^{-1}(\mSpace_\kap^+)$.
Therefore by \eqref{e:ffi_2}-\eqref{e:dg} we have
\begin{equation*}
\begin{aligned}
        \partial_a \ff_2 &= -\bigg[  \frac{2\beta}{\kap} + \frac{\beta(\kap - 1)}{\kap} \Big( a\partial_a g(a, b, c) + g(a, b, c)  + b \partial_af(a, b, c) \Big)\bigg] = - \frac{2\beta}{\kap}, \\
        \partial_b \ff_2 &= - \bigg[ 2\alpha + \frac{\beta(\kap - 1)}{\kap}\Big( \partial_bg(a, b, c) + b\partial_bf(a, b, c) + f(a, b,c)\Big)\bigg] =  - 2 \alpha,\\ 
        \partial_c 
        \ff_2 &= \frac{\beta(\kap - 1)}{\kap}\bigg[a \partial_c g(a, b, c) + b \partial_c f(a, b, c) + 1 \bigg] = -\frac{2\beta(\kap - 1)}{\kap}.
\end{aligned}
\end{equation*}
Thus $\ff_2$ is also uniformly Lipschitz. Finally, by \eqref{e:ffi_2}-\eqref{e:dg}, we have
\begin{equation*}
\begin{aligned}
        \partial_a \ff_3 &= - \frac{\beta(\kap -1)}{\kap}\bigg[ b \partial_a g(a, b, c) + c \partial_a f(a, b, c)\bigg] = - \frac{2\beta(\kap - 1)}{\kap} f(a,b,c)g(a,b,c) \\
        \partial_b 
        \ff_3 &= - \frac{\beta}{\kap} - \frac{\beta(\kap -1)}{\kap}\bigg[ g(a, b, c) + b \ \partial_bg(a, b, c) + c \partial_b f(a, b, c)\bigg] = - \frac{\beta}{\kap} + \frac{\beta(\kap - 1)}{\kap} \big(f(a, b, c)^2 + g(a, b, c)^2\big) \\
        \partial_c \ff_3 &= - \alpha - \frac{\beta(\kap - 1)}{\kap}\bigg[ b\partial_c g(a, b, c) + c \partial_c f(a, b, c) + f\bigg] = -\alpha - \frac{2\beta(\kap - 1)}{\kap}f(a, b, c)
\end{aligned}
\end{equation*}
The previous display, together with the aforementioned boundedness of $f$ and $g$ on $\Mat_\kap^{-1}(\mSpace_\kap^+)$, implies that $\ff_3$ is uniformly Lipschitz as well. Thus $\sfF$ is uniformly Lipschitz.
\end{proof}

\begin{proof}[Proof of \eqref{e:mRelations}]
Recall that we want to show:
    \begin{equation*}
    \begin{aligned}
        &\sfQ \sfJ = \sfJ (\sfI  - \sfQ) = \sfP, &\quad &\sfQ \sfK = \sfK\sfQ = \sfP^* \sfK = \sfK\sfP =0,  &\quad &\sfJ\sfP = \sfP^* \sfJ = \sfK + (\sfI  - \sfQ), \\
        &\sfP \sfJ = \sfJ\sfP^* = \kappa \sfQ, &\quad &(\sfP\sfK)^* = \sfK\sfP^* = (\kap - 1)\sfP^*,  &\quad &(\sfI  - \sfQ) \sfK = \sfK(\sfI  - \sfQ) = \sfK.        
    \end{aligned}
    \end{equation*}
In entry form, we have
\begin{equation*}
    \sfQ_{ij} = \delta_{i0}\delta_{j0}, \quad \sfP_{ij} = \delta_{i0}\sum_{r = 1}^\kap \delta_{rj}, \quad \sfR_{ij} = 1, \quad i,j \in \{0, \ldots, \kap\}.
\end{equation*}
and that $\sfJ = \sfP + \sfP^*$ and $\sfK = \sfR - \sfI - \sfJ$. In the following, we use Einstein notation for indices. 

First, we show $\sfQ \sfJ = \sfJ(\sfI -\sfQ) = \sfP$. We have
\begin{equation}
    (\sfQ\sfP)_{ij} = \sfQ_{il} \sfP_{lj} = \delta_{i0}\delta_{l0}\bigg(\delta_{l0}\sum_{r = 1}^\kap\delta_{jr}\bigg) = \delta_{i0}\sum_{r = 1}^\kap\delta_{jr} = \sfP_{ij}.
    \label{e:QP = P}
\end{equation}    
Similarly, we have
\begin{equation}
    (\sfQ\sfP^*)_{ij} = \sfQ_{il} \sfP_{lj} = \delta_{i0}\delta_{l0} \bigg(\delta_{j0}\sum_{r = 1}^\kap\delta_{lr}\bigg) = \delta_{i0}\delta_{j0} \bigg(\sum_{r = 1}^\kap \delta_{l0}\delta_{lr}\bigg) = 0.
\end{equation}
Recall $\sfJ = \sfP + \sfP^*$. Then we have
\begin{equation*}
    \sfQ\sfJ = \sfQ(\sfP +\sfP^*) = \sfP.
\end{equation*}
Since $(\sfQ\sfJ)^* = \sfJ\sfQ$, we have $\sfQ\sfJ + \sfJ \sfQ = \sfP + \sfP^* = \sfJ$, which verifies the second equality.

Next we compute the products of $\sfP$ and $\sfP^*$. Notice that we have
\begin{equation*}
    (\sfP\sfP)_{ij} = \sfP_{il}\sfP_{lj} =  \bigg( \delta_{i0} \sum_{r = 1}^\kap \delta_{lr}\bigg) \bigg( \delta_{l0} \sum_{r'= 1}^\kap \delta_{jr'}\bigg) = \delta_{i0}\sum_{r' = 1}^\kap \delta_{j r'}\bigg( \sum_{r = 1}^\kap \delta_{l0}\delta_{lr}\bigg) = 0.\label{e:PP}
\end{equation*}
Similarly, we have
\begin{equation}
    (\sfP\sfP^*)_{ij} = \sfP_{il}\sfP_{jl} = \bigg( \delta_{i0} \sum_{r = 1}^\kap \delta_{lr}\bigg) \bigg( \delta_{j0} \sum_{r'= 1}^\kap \delta_{lr'}\bigg) = \delta_{i0}\delta_{j0} \sum_{r = 1}^\kap \sum_{r' = 1}^\kap\delta_{lr} \delta_{lr'} = \kap \sfQ_{ij}.
    \label{e:PP*}
\end{equation}
Finally, we have
\begin{equation}
    (\sfP^*\sfP)_{ij} = \sfP_{li}\sfP_{lj} = \delta_{l0}\bigg(\sum_{r = 1}^\kap\delta_{ri}\bigg) \delta_{l0} \bigg(\sum_{r' = 1}^\kap \delta_{r'j}\bigg) = \sum_{r = 1}^\kap
\sum_{r' =1}^\kap \delta_{ri}\delta_{r'j} = (1 - \delta_{j0})(1-\delta_{i0})\label{e:P*P}\end{equation}
Then the above identities show that we have $\sfP\sfJ = \sfP(\sfP + \sfP^*) = \kap \sfQ$. By \eqref{e:PP} we have $\sfJ\sfP = \sfP^*\sfP$, and notice by \eqref{e:P*P} that $\sfP^*\sfP = \sfR - \sfJ - \sfQ = \sfK + (\sfI - \sfQ).$ Therefore, we have also proven that $\sfJ\sfP = \sfK + (\sfI - \sfQ)$.

Next, we show $\sfQ\sfK = 0$. First, we compute
\begin{equation*}
    (\sfQ \sfR)_{ij} = \sfQ_{il}\sfR_{lj} = \delta_{i0} = (\sfQ + \sfP)_{ij}
\end{equation*}
Similarly, we have from \eqref{e:QP = P} that $ (\sfQ + \sfQ\sfJ) = \sfQ + \sfP.$ Therefore, we have
\begin{equation*}
\sfQ\sfK = \sfQ(\sfR - (\sfI + \sfJ)) = 0.
\end{equation*}

Now, we establish $\sfP\sfK = (\kap - 1) \sfP$. We have 
\begin{equation*}
    (\sfP\sfR)_{ij} = \sfP_{il} = \kap \delta_{i0} = \kap( \sfQ + \sfP)_{ij}
    \label{e:PR}
\end{equation*}
Therefore by \eqref{e:PP} and \eqref{e:PP*} we have $$\sfP \sfK = \sfP(\sfR - \sfI - \sfJ) = \sfP\sfR - \sfP - \sfP\sfP - \sfP\sfP^* = \kap(\sfQ + \sfP) - \sfP - \kap \sfQ = (\kap - 1) \sfP.$$ 

We show $\sfK \sfP = 0$. We have
\begin{equation}
    (\sfR\sfP)_{ij} = \sum_{r = 1}^\kap\delta_{jr}.
\end{equation}
Therefore by \eqref{e:P*P} and the above display we have
\begin{equation*}
\begin{aligned}
    (\sfK\sfP)_{ij} &= (\sfR \sfP)_{ij} - \sfP_{ij} - (\sfJ\sfP)_{ij} \\ &= \sum_{r = 1}^\kap \delta_{jr} - \delta_{i0}\sum_{r = 1}^\kap
\delta_{jr} - \bigg(\sum_{r = 1}^\kap \sum_{r' = 1}^\kap \delta_{jr}\delta_{ir'}\bigg) \\ &= (1 - \delta_{i0})\sum_{r = 1}^\kap \delta_{jr} - \bigg(\sum_{r = 1}^\kap \sum_{r' = 1}^\kap \delta_{jr}\delta_{ir'}\bigg) \\ &= 0.
\end{aligned}
\end{equation*}
Finally, notice that $\sfI - \sfQ$ is the identity matrix on $\kap$ elements (with a $0$ in the $0$-th entry), and that $\sfK$ has all zeros in the first row and first column. This implies $(\sfI- \sfQ) \sfK = \sfK$.
\end{proof}

\bibliographystyle{abbrv}
\bibliography{refs}

\end{document}